\documentclass{amsart}

\usepackage{max}
\usepackage{xfrac}
\usepackage{calrsfs,bbm}

\usetikzlibrary{knots,patterns,decorations.pathreplacing}

\title{Disordered arcs and Harer stability}
\author{Oscar Harr, Max Vistrup, and Nathalie Wahl}
\date{\today}

\newcommand*{\M}{{{\mathbf M}_2}}
\newcommand{\braidGrpd}{\mathbf B}
\newcommand{\BB}{\braidGrpd}
\newcommand{\E}{E}

\usepackage[frame,cmtip,arrow,matrix,line,graph]{xy}

\newtheorem{ThA}{Theorem}

\newcommand{\hash}{\mathbin{\text{\normalfont \texttt{\#}}}}

\newcommand{\inc}{\hookrightarrow}

\newcommand{\rar}{\longrightarrow}

\newcommand{\Ga}{\Gamma}

\newcommand{\s}{\sigma}

\newcommand{\Si}{\Sigma}

\newcommand{\x}{\times}

\newcommand{\id}{\operatorname{id}}

\newcommand{\Homeo}{\f{Homeo}}

\tikzset{
  vertlabl/.style={anchor=north, rotate=90, inner sep=1mm}L
}

\title{Disordered arcs and Harer stability}

\begin{document}

\begin{abstract}

  We give a new proof of homological stability
  with the best known isomorphism range
  for mapping class groups of surfaces with respect to genus.
  The proof uses the framework of Randal-Williams--Wahl and Krannich
  applied to disk stabilization in the category of bidecorated surfaces,
  using the Euler characteristic instead
  of the genus as a grading.
  The monoidal category of bidecorated surfaces does not admit a braiding,
  distinguishing it from previously known settings for homological stability.
  Nevertheless, we find that it admits a suitable Yang--Baxter element,
  which we show is sufficient structure for homological stability arguments.
  

  \end{abstract}

\maketitle

\section{Introduction}
\label{sec:intro}
Let $S_{g,r}^s$ be a surface of genus $g$
with $r$ boundary components and $s$ punctures.
The mapping class group $\Gamma(S_{g,r}^s):=\pi_0\Homeo (S_{g,r}^s \textrm{ rel } \del S)$
of $S$ satisfies {\em homological stability}:
the homology group $H_i(\Gamma(S_{g,r}^s);\Z)$
is independent of $g$ and $r$ when $g$ is large relative to $i$.
This stability result was originally proved by Harer in \cite{harer85},
and later improved by Ivanov, Boldsen and Randal-Williams
\cite{Iva89,boldsen12,RW16},
see also \cite{Har93,wahl-handbk,HatVog17,GKRW19MCG}.
We recast the result here as a stability theorem
in the category of {\em bidecorated surfaces},
and give a new proof of the best know stability range using the most
straightforward inductive argument originally designed by Quillen, and
formalized in \cite{RWW17,krannich19}.
Our proof at the same time illustrates how little is needed to run the stability
machines of these two papers.

\smallskip

Our main stability result is the following, recovering precisely the
ranges of \cite[Thm 1]{boldsen12} and \cite[Thm 7.1 (i),(ii)]{RW16}:

\begin{ThA}\label{thmintro:stab1}
  Let $S_{g,b}^s$ be a surface of genus $g\ge 0$, with $r\ge 1$
  marked boundary components  and $s\ge 0$ punctures, and let 
  $\Ga(S_{g,r}^s)=\pi_0\Homeo (S_{g,r}^s\textrm{ rel }\del S)$ denote its mapping class group.
  The map 
  $$
  H_i(\Gamma(S_{g,r}^s);\Z )\to H_i(\Gamma(S^s_{g,r+1}); \Z)
  $$
  induced by gluing a pair or pants along one boundary component is always injective, 
  and  an isomorphism when $i\leq\frac{2g} 3$, and the map 
  $$
  H_i(\Gamma(S_{g,r+1}^s);\Z )\to H_i(\Gamma(S^s_{g+1,r}); \Z) $$
  induced by gluing a pair of pants along two boundary components is an  epimorphism
  when $i\leq\frac{2g+1} 3$
  and an  isomorphism when $i\leq\frac{2g-2} 3$.
\end{ThA}

Combining the two maps in the theorem gives a genus stabilization that
is  known to be close to  optimal by a computation of
Morita \cite{Morita}  and low dimensional computations, see
Remarks~\ref{rem:optimal1} and~\ref{rem:optimal2}.
While we do not
know whether the two
ranges in the above statement can be individually improved, it is
remarkable that three rather different proofs (those of
Boldsen \cite{boldsen12},  Randal-Williams \cite{RW16}, and ours) end up with the
exact same ranges.

\medskip

A particular feature of our proof is that the two maps occurring in
the theorem will be for us ``the same map'',
namely a disk stabilization
in the category $\M$ of {\em bidecorated surfaces}.
A bidecorated surface is
a surface $S$ with two marked intervals  $I_0,I_1$ in its boundary. 
The two intervals may lie on the same or on different boundary components.
Morphisms in $\M$ are mapping classes, i.e.~isotopy classes of homeomorphisms,
and $\M$ admits a monoidal structure $\hash$
defined by identifying the marked intervals in pairs.

Our main example of a bidecorated surface will be the bidecorated disk $D$.
As shown in Lemma~\ref{lem:surface-type},
taking sums of the disk with itself in $\M$ produces surfaces of any genus:
$D^{\hash 2g+1}$ is a surface $S_{g,1}$ of genus $g$ with a single boundary component,
while $D^{\hash 2g+2}$ is a surface $S_{g,2}$ of genus $g$ with two boundary components,
each containing a marked interval.
To obtain any surface $S_{g,r}^s$ with $r\ge 1$,
we will consider the object $S\hash D^{\hash 2g}$ in $\M$,
for $S=S_{0,r}^s$ a genus 0 surface with $r$ boundary components and $s$ punctures.
Now the maps in Theorem~\ref{thmintro:stab1}
are precisely  the disk stabilization maps in $\M$:
$$
\Aut_\M (S\hash D^{\hash 2g}) \xrightarrow{\hash D} \Aut_\M (S\hash D^{\hash 2g+1}) \xrightarrow{\hash D} \Aut_\M (S\hash D^{\hash 2g+2})
$$
for these particular choices of surfaces. 

Theorem~\ref{thmintro:stab1} is thus the statement that disk stabilization $\hash D$ in $\M$ induces isomorphisms on
the homology of these automorphism groups in a range.
 We show in the present paper that this result can be obtained as a
direct application of the main result of  \cite{krannich19},
from which an additional stability statement with twisted coefficients
automatically follows. We start by stating this additional result.

\subsection*{Twisted coefficients}
Fix $r\ge 1 $ and $s\ge 0$. In our setting, a coefficient system $F$ for the mapping class group $\Ga(S_{g,r}^s)$ is a collection of $\Z[\Ga(S_{g,r}^s)]$-modules
$F_{2g}$ and  $\Z[\Ga(S_{g,r+1}^s)]$-modules $F_{2g+1}$ for each
$g\ge 0$, together with maps
$$
F_{n}\to F_{n+1}
$$
equivariant with respect to the disk stabilization and satisfying that
a certain Dehn twist 
acts trivially on the image of $F_{n}$ in $F_{n+2}$
under double stabilization (see Definition~\ref{def:coef-system}).
Given a coefficient system,
one can define a notion of degree;
a constant coefficient systems has degree 0
and for example the coefficient system
$F_{2g+i }=H_1(S_{g,r+i}^s;\Z)^{\otimes k}$, $i\in\lbrace 0,1\rbrace$, has degree $k$
(see Example~\ref{exmp:coef-systems}). 

We obtain the following twisted stability result:

 \begin{ThA}\label{thmintro:stab2}
      Let $\Ga(S_{g,r}^s)$ be as in Theorem~\ref{thmintro:stab1}, and $F$ be a coefficient system of degree $k$.
      The stabilization map
      \begin{equation*}
        \label{eq:twisted-stab1}
        H_i(\Gamma (S_{g,r}^s);F_{2g} )\to H_i(\Gamma (S^s_{g,r+1}); F_{2g+1} ) 
      \end{equation*}
      is an epimorphism
      for $i\leq \frac{2g-3k-2} 3$
      and  an isomorphism for $i\leq\frac{2g - 3k - 5} 3$, and
      the map
      \begin{equation*}
        \label{eq:twisted-stab2}
        H_i(\Gamma (S_{g,r+1}^s);F_{2g+1} )\to H_i(\Gamma (S^s_{g+1,r}); F_{2g+2} )
      \end{equation*}
      is an epimorphism
      for $i\leq\frac{2g-3k-1}{3}$
      and  an isomorphism for $i\leq\frac{2g-3k-4} 3$.
      In these bounds, $3k$ can be replaced by $k$  if $F$ is in addition split in the sense of Definition~\ref{def:coef-system}. 
   \end{ThA}

Stability theorems for mapping class groups with twisted coefficients
can be found in the work of  Ivanov, Boldsen, Randal-Williams--Wahl,
and Galatius--Kupers--Randal-Williams \cite{boldsen12,IvanovTwisted,RWW17,GKRW19MCG}. 
The results are not easy to compare as the types of coefficient
system that are permitted depend on the paper,
but some classical examples such as the one described above fit
all frameworks (see Remarks~\ref{rmk:different-category}
and~\ref{rem:optimal2} for more details).

\subsection*{Braided action and Yang--Baxter operators}
We want to obtain Theorems~\ref{thmintro:stab1}
and~\ref{thmintro:stab2} as consequences of 
Theorems A and C of \cite{krannich19}.  For this,  we first have to show that disk
stabilization in the monoidal category $(\M,\hash)$ comes from  an
action of a braided monoidal groupoid.

Let $\BB$ denote the  groupoid of braid groups, with
object the natural numbers and the braid group $B_n$ as automorphisms
of $n$.  We will construct an action of $\BB$ on $\M$ using
an appropriate {\em Yang--Baxter operator} in $\M$:  
The sum of bidecorated disks $D\hash D$ in $\M$ is a cylinder, whose
mapping class group is an infinite cyclic group generated
by the Dehn twist $T$ along the core circle of the cylinder. It turns
out that this morphism $T\in
\Aut_{\M}(D\hash D)$ is a Yang--Baxter operator in $\M$,
in the sense that it satisfies the equation
$$(T\hash 1)(1\hash T)(T\hash 1)=(1\hash T)(T\hash 1)(1\hash T)$$
in $\Aut_{\M}(D^{\hash 3})$. The same holds for the inverse twist
$T^{-1}$, that will turn out more convenient for us. 
As explained in Section~\ref{sec:YB},
we  get an associated strong monoidal functor
$\BB\to\M$ taking the object $n$ to $D^{\hash n}$. 
The corresponding homomorphism $B_n\to\text{Aut}_\M (D^{\hash n})$
can be identified with the geometric embedding in the sense of \cite{Waj99},
associated to the chain of curves $a_1,\dots,a_{n-1}$ in
$$
D^{\hash n} = D\hash D\hash \dots \hash D,
$$
where the $i$th curve $a_i$ is the core circle in the $i$th cylinder
$D\hash D$ in the above sum,
see Lemma~\ref{lem:chain} and Example~\ref{exmp:mcg-YB}.

The strong monoidal functor $\BB\to\M$ from above endows $\M$
with the structure of an $\E_1$-module over the braid groupoid $\BB$,
and since the latter is braided monoidal,
we can apply the results of \cite{krannich19} to study disk stabilization in $\M$.

\begin{remark}
  Homological stability frameworks such as \cite{RWW17,krannich19,GKRW19MCG}
  require an $E_2$-algebra, or the weaker structure of $E_1$-module
  over an $E_2$-algebra, as input. 
  This is a priori a lot of data, and it may be that the most natural choice
  in a given context simply does not admit an $E_2$-structure.
  This turns out to be the case for the monoidal category of bidecorated surfaces
  $\M$: 
  In the context of categories,
  $E_2$-structures are given by braided monoidal structures and we show in Section~\ref{sec:notbraided}
  that even the full monoidal subcategory of $\M$ generated by
  our stabilizing object, the disk $D$,
  does not admit a braiding. 
  This distinguishes our situation from most previous examples of homological stability.

  On the other hand, it does not take much
  to equip a given monoidal category $\mathcal X$
  with the structure of an $E_1$-module over a braided monoidal category.
  In fact, as shown in Section~\ref{sec:YB}, any Yang--Baxter operator in $\mathcal X$ determines
  a strong monoidal functor $\BB\to\mathcal X$ from the braid groupoid
  $\BB$, and thus endows $\mathcal X$
  with the structure of an $E_1$-module over $\BB$.
  This perspective also makes sense if $\mathcal X$ itself acts
  on a category $\mathcal M$, and one is interested in the stabilization
  $$
  \mathcal M\xrightarrow{\oplus X}\mathcal M\xrightarrow{\oplus X}\cdots
  $$
induced by acting with an object $X$ of $\mathcal X$ admitting a
Yang-Baxter operator $\tau\in\text{Aut}_{\mathcal X}(X\oplus X)$. 
  The category $\mathcal M$ becomes this way likewise a module over $\BB$,
  where the object $n$ of $\BB$  acts on  $A\in\mathcal M$ via $A\oplus n = A\oplus X^{\oplus n}$.
\end{remark}

\subsection*{Disordered arcs}
Given a category $\mathcal M$ as above, with the structure of an $E_1$-module over a
monoidal category $\mathcal X$
with a distinguished Yang--Baxter operator $(X,\tau )$,
such that acting by $X$ satisfies a certain injectivity property
(see Proposition~\ref{prop:Mmonoid}),
the main result of \cite{krannich19} implies that 
homological stability for stabilization with $X$
is controlled by the connectivity of certain {\em complexes of
  destabilizations}.
In the category of bidecorated surfaces $\M$,
stabilizing with the bidecorated disk $D$ corresponds
homotopically to attaching an arc,
and we show in Proposition~\ref{prop:Wiso}
that the relevant complex of destabilizations
for stabilizing a surface $S$ with a disk $n$ times 
identifies with the
``disordered arc complex''\footnote{We called those {\em disordered} arcs because it is the opposite ordering convention than the one used in the ``ordered arc complex'' of \cite{RW16}.}
associated to the surface $S\hash D^{\hash n}$. This is
a simplicial complex whose vertices are isotopy classes of 
non-separating arcs in the surface
with endpoints $b_0 = I_0(\rfrac 1 2)$
and $b_1 = I_1(\rfrac 1 2)$,
and where a collection of isotopy classes forms a simplex
if the classes can be represented by arcs that are disjoint away from
the endpoints, are jointly non-separating,
and such that the arcs have the same ordering at $I_0$ and $I_1$.

Writing $\mathcal D^\nu(S_{g,r}, b_0, b_1)$ for the disordered arc
complex of a surface $S_{g,r}$ with marked points $b_0$ and $b_1$ in
$\nu=1$ or $\nu=2$ boundary components, the main ingredient of our proof of homological stability is the
following connectivity result:

\begin{ThA}(Theorem~\ref{thm:disord-cnt})\label{introthm:C}
   The disordered arc complex $\mathcal D^\nu(S_{g,r}, b_0, b_1)$ is $\(\frac{2g +\nu - 5} 3\)$-connected.
  \end{ThA}

  \begin{remark}
    It is conjectured in \cite[Conj C]{RWW17} that the complex of destabilizations is highly connected if and only
    if stability holds with all appropriate twisted coefficients.
    The slope $2/3$ bounds in Theorems~\ref{thmintro:stab1} and
    \ref{thmintro:stab2} is precisely dictated by 
    the same slope $2/3$ in Theorem~\ref{introthm:C} in the connectivity of the arc
    complex, which is the complex of destabilizations in that case.
    This connectivity bound is best possible among linear bounds as a
    better bound would prove an incorrect stability statement, see
    Remark~\ref{rem:optimal1}.
    \end{remark}

 \subsection*{Organization of the paper.}
 In Section~\ref{sec:high-cnt}
 we prove the high connectivity of the disordered arc complex.
 In Section~\ref{sec:category} we define the monoidal category of
 bidecorated surfaces $(\M,\hash)$, as well as the action of the braid groupoid
 $\BB$ on this category.
 In Section~\ref{sec:stab}, we show Theorems A and B by showing that
 the disordered arc complex agrees with the complex of
 destabilizations, and applying the main result of 
 \cite{krannich19}. 
 Finally, in Section~\ref{sec:braiding} we explain the relationship between
 homological stability and Yang--Baxter operators, and show the
 non-braidedness of the category of bidecorated surfaces.

    \subsection*{Acknowledgements} 
The first and third authors were partially supported by the Danish
National Research Foundation through the Copenhagen Centre for
Geometry and Topology (DRNF151), and the third author by the European Research Council (ERC) under the European Union’s Horizon 2020 research and innovation programme (grant agreement No. 772960).

\section{High connectivity of the disordered arc complex}
\label{sec:high-cnt}

In this section, we prove that the disordered arc complex is highly
connected. It will be defined as a subcomplex of the following simplicial complex of non-separating arcs: 

\begin{definition}\label{def:nonseparating}
  Let $S$ be an orientable surface\footnote{
  By {\em surface} we mean a topological $2$-manifold $S$ which is compact except
  for a finite number of punctures, i.e. there is a compact topological $2$-manifold
  $\overline S$ and an embedding $i\colon S\incl\overline S$ so that
  $\overline S\setminus i(S)$ is a (possibly empty) finite union of points.
}
  with nonempty boundary, and
    let $b_0, b_1$ be distinct points in $\del S$.
    The \emph{complex of non-separating arcs} $\mathcal B(S, b_0, b_1)$ is the simplicial complex
        whose $p$-simplices are collections of $p + 1$ distinct isotopy classes
            of arcs between $b_0, b_1$
            that admit representatives
            $
                a_0, \ldots, a_p
            $
            such that
            \begin{enumerate}
                \item\label{def:nonseparating:1}
                $a_i \cap a_j = \{b_0, b_1\}$ for each $i \neq j$ and
                \item\label{def:nonseparating:2}
                $S - (a_0 \cup \cdots \cup a_p)$ is connected.
            \end{enumerate}
          \end{definition}

For convenience, we will add a superscript $\mathcal
B^\nu(S, b_0, b_1)$ to the notation of the complex, with 
        $\nu=1$ indicating that $b_0, b_1$ lie on the same boundary component and
        $\nu=2$ indicating that they do not. 

        \smallskip

        Note that the orientation of the surface defines orderings of the arcs $a_0,\dots,a_p$ representing a simplex at both $b_0$ and $b_1$.

\begin{definition}\label{def:disord}
    Let $(S, b_0, b_1)$ be as before.
    The \emph{disordered arc complex} is
        the subcomplex
        $
            \mathcal D^\nu(S_{g, r}, b_0, b_1)
            \subseteq
            \mathcal B^\nu(S, b_0, b_1)
        $
        consisting of those simplices $\sigma$
        that admit arc representatives $a_0, \ldots, a_p$, again subject to
            \cref{def:nonseparating:1},
            \cref{def:nonseparating:2}, satisfying in addition
            \begin{itemize}
                \item[(c)] the ordering of the arcs at $b_0$ agrees with the ordering of the arcs at $b_1$. 
            \end{itemize}
\end{definition}

The name ``disordered'' was chosen to contrast with the pre-existing {\em ordered arc complex} used by Ivanov \cite{Iva89} in the case $\nu=1$ and Randall-Williams \cite{RW16} in their proofs of homological stability for the mapping class group of surfaces; the ``ordered'' version is also a subcomplex of the $\mathcal B^\nu(S, b_0, b_1)$, but with the requirement that the order of the arcs at $b_1$ is reversed compared to the order at $b_0$. 
Fixing an ordering condition has the effect that the action of the mapping
class group is transitive on the set of $p$-simplices for every $p$,
see \cite[Lem 3.2]{harer85}. The ordered and disordered arc complexes  represent  the two extremes of how fast the genus of the surface decreases when cutting along larger and larger simplices: for the ordered arc complex, the genus goes down as fast as possible, essentially every time one removes an arc, while for the disordered arc complex, the genus goes down as slow as possible, only every other time:

\begin{proposition}\label{prop:type-of-cut-surface}
  For a $p$--simplex  $\sigma = \langle a_0,\dots ,a_p\rangle\in\mathcal D^\nu (S_{g,r},b_0,b_1)$, 
  the surface $S_{g,r}\setminus \sigma$ obtained by removing a tubular neighborhood of $a_i$ for each $i$ has genus $g'$ with $r'$ boundary components for 
   $$  g' = g-\floor{\frac{p+3-\nu} 2} \ \ \ \textrm{and}\ \ \ 
  r' =
  \begin{cases}
    r+(-1)^\nu,&\text{if } p \text{ is even,} \\
    r&\text{else}. 
  \end{cases}
  $$
\end{proposition}
\begin{proof}
The number of boundary components $r'$ can be obtained by a direct
inductive computation, with the genus $g'$ then deduced using the
Euler characteristic.  The computation is a special case of \cite[Prop 2.11]{boldsen12}, applied to the case where the permutation $\alpha$ is the inversion $[p\,(p-1)\dots 0]$, once one computes that the genus $S(\alpha)$ of a neighborhood of the arcs is $\lfloor\frac{p+2-\nu}{2}\rfloor$, e.g.~using Corollary 2.15 of the same paper. 
\end{proof}

The complex   $\mathcal B^\nu(S, b_0, b_1)$ is known to be  $(2g + \nu - 3)$-connected. (This was first stated in \cite{harer85}; see \cite[Thm 3.2]{Wah08} or \cite[Thm 4.8]{wahl-handbk} for a complete proof.)  
We will here  use this fact to deduce that $\mathcal D^\nu(S_{g, r}, b_0, b_1)$ is also highly-connected.
While the ordered arc complex is $(g-2)$-connected \cite[Thm A.1]{RW16},
the following result shows that the disordered arc complex is only slope
$\frac{2}{3}$ connected with respect to the genus,
despite being $\sim 2g$-dimensional.

\begin{theorem}\label{thm:disord-cnt}
   The disordered arc complex $\mathcal D^\nu(S_{g,r}, b_0, b_1)$ is $\(\frac{2g +\nu - 5} 3\)$-connected.
\end{theorem}

 To prove the result, we use essentially the same argument as the one given in \cite{RW16} in the ordered case. 

\begin{proof}
    Let $S = S_{g, r}$.   In the case $g = 0$,
        the statement for $\mathcal D^1(S)$ is vacuous,
        and for $\mathcal D^2(S)$ it states that the complex  is $(-1)$-connected, i.e.~nonempty, which holds as any arc in the surface connecting $b_0$ and $b_1$ defines a vertex in $\mathcal D^2(S)$. 
      We prove the remaining cases by induction on $g$.

    Let $g > 0$. 
    Suppose we are given 
        $
            f : \del D^{k + 1} \to \mathcal D^\nu(S, b_0, b_1)
        $
        for some $k \leq (2g + \nu - 5)/ 3$.
    We wish to exhibit a nullhomotopy of this map.
    Since $(2g + \nu - 5) / 3 \leq 2g + \nu - 3$, Theorem 3.2 in \cite{Wah08} 
     enables us to choose a map $\hat f$
        such that the outer diagram
        \begin{equation}\label{lifting-prob}
            \begin{tikzcd}
                \del D^{k + 1}
                \ar[d, hook]
                \ar[r, "f"]
                    & \mathcal D^\nu (S, b_0, b_1)
                    \ar[d, hook]
                \\
                D^{k+1}
                \ar[r, "\hat f"]
                \ar[ru, dotted]
                    & \mathcal B^\nu(S, b_0, b_1),
            \end{tikzcd}
        \end{equation}
        commutes.
    Using PL-approximation,
        we may assume that $\hat f$ and $f$ are simplicial
            with respect to some PL-triangulation of $D^{k+1}$.
    We will repeatedly replace $\hat f$ until the dotted arrow exists,
        thereby giving the desired nullhomotopy.

         Write $<_0$ and $<_1$ for the anti-clockwise orderings at $b_0$ and $b_1$.  We call a $p$-simplex $\sigma$ in $D^{k+1}$  \emph{regular bad} if 
        $\hat f(\sigma) = \<a_0, \cdots, a_{p'}\>$,
        indexed in such a way that $a_0 <_0 \dots <_0 a_{p'}$ are anticlockwise at $b_0$,
       and 
            there is $j > 0$ with $a_j <_1 a_0$  at $b_1$. Here $p'\le p$ is the dimension of the image simplex $\hat f(\s)$, and $p'\geq 1$ if $\sigma$ is regular bad. 
    This condition is \q{dense} in the sense
        that any simplex $\sigma$ in $D^{k+1}$ with image not included in $\mathcal D^\nu(S, b_0, b_1)$
            must contain a regular bad simplex as a face.
    Thus it suffices to give a procedure for exchanging $\hat f$ with a map
        having strictly fewer regular bad simplices, 
        while maintaining commutativity of the outer diagram~\cref{lifting-prob}.

    Let $\sigma$ be a regular bad simplex of $D^{k+1}$ of maximal dimension $p$ and consider its link $\Lk\sigma\subset D^{k+1}$.
    Maximality implies that $\hat f | _{\Lk \sigma}$ factors as
        \[\hat f | _{\Lk \sigma}\colon 
            \Lk \sigma
            \to
            \mathcal D^\mu(S \setminus \hat f(\sigma), b_0', b_1')
            \to
            \mathcal D^\nu(S,b_0,b_1)
           \subseteq \mathcal B^\nu(S, b_0,b_1), 
        \]
  where  $S \setminus \hat f (\sigma)$ is the closure of the surface obtained from $S$
        by cutting out the collection of arcs $\hat f(\sigma)$, and 
        $b_0'$ and $b_1'$ in $S \setminus \hat f (\sigma)$ are the first copies of $b_0$ and $b_1$ in the cut surfaces as depicted in Figure~\ref{fig:regularbad}.
        \begin{figure}
          \def\svgwidth{0.8\textwidth}
          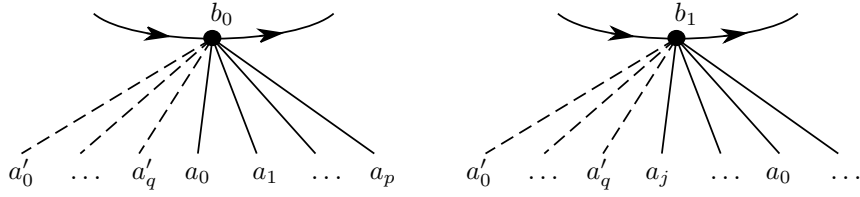
          \caption{Maximal regular bad simplex $\{a_0,\dots,a_p\}$ and simplex $\{a'_0,\dots,a'_q\}$ in its link.}\label{fig:regularbad}
        \end{figure}
Indeed, suppose that $\tau \in \Lk\sigma$ and write
        $
            \hat f(\tau) = \<a_0', \cdots, a_q'\>. 
        $
        If $a_0\le_0a_i'$ at $b_0$ for any $a_i'$, then the simplex $\sigma*\<a_i'\>$ is regular bad of a larger dimension, contradicting maximality. So we must have $a_i'<_0a_0$ for each $i$, i.e.~the arcs of $\tau$ are at $b_0'$ in the cut surface. Now we must also have that each $a_i'<_1a_0$ as otherwise $\sigma*\<a_i'\>$ would again be regular bad. 
 Finally,  maximality of $\s$ would also be contradicted if the orderings of the arcs $a_0', \cdots, a_q'$ does not agree at $b_0$ and $b_1$ as,  if $a_i'<_0 a_j'$ with $a_j'<_1a_i'$ for some $i,j$, then $\s*\<a'_i,a'_j\>$ would again be regular bad, of larger dimension.   
    Thus $\hat f(\tau)$ must be disordered
        and, after cutting the surface at the arcs of $\s$,  can be viewed as a simplex of
        $
            \mathcal D^\nu(S \setminus \hat f(\sigma), b_0', b_1'). 
            $

The link $\Lk(\s)$ is a simplicial sphere $S^{k-p}\subset D^{k+1}$.  We want to show that the map $\hat f|_{\Lk (\sigma)}$ extends to a simplicial map
            \begin{equation}\label{equ:F}
                F :
                D^{k - p + 1}
                \to
                \mathcal D^\mu(S - \hat f(\sigma), b_0', b_1')
                \to
                \mathcal D^\nu (S, b_0, b_1)
                \subseteq
                \mathcal B(S, b_0, b_1)
            \end{equation}
            for $D^{k - p + 1}$ a disk with some PL-structure extending that of $\Lk (\sigma)$.
This will follow if we can show that the complex
        $
            \mathcal D^\mu (S\setminus\hat f (\sigma ),b_0',b_1')
        $
        is at least $(k-p)$-connected. Note that necessarily have $g(S\setminus\hat f (\sigma )) < g$ as $f(\s)$ is a non-separating $p'$-simplex with $p'\ge 1$. Hence we can use our induction hypothesis.
        We consider the cases $\nu = 1$ and $\nu = 2$ separately.

\smallskip
        
    \para{Case 1: $\nu=1$}
  We have that 
        $
      g(S\setminus\hat f (\sigma ))\ge g-p'-1\ge g-p-1,
        $
as removing $p'+1$ arcs  reduces the genus by at most $p'+1\le p+1$.
 Hence by induction we have that $\mathcal D^\mu (S\setminus\hat f (\sigma )
    , b_0',b_1')$ is at least $(\frac{2(g-p-1)-4}{3})$-connected, using also that $\mu\ge 1$.
    If  $p\geq 2$, we have 
    $$k-p\leq \frac{2g-4}{3} - p = \frac{2g-3p-4}{3}  \leq \frac{2(g-p-1)-4}{3}.$$
    For $p=p'=1$, note that  $b_0',b_1'$ necessarily lie in different boundary components, so that $\mu=2$ in that case. (See Figure~\ref{fig:boundaries}.) Hence
  in that case  $\mathcal D^\mu (S\setminus\hat f (\sigma ), b_0',b_1')$ is  
    $(\frac{2(g-2)-3}{3})$-connected, and 
   $$k-1\leq \frac{2g-4}{3} - 1 = \frac{2g-7}{3} = \frac{2(g-2)-3}{3}.$$
   so we get the desired extension in both subcases.
   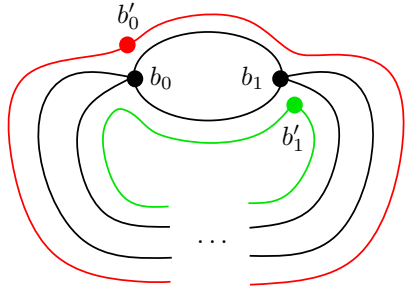
\begin{figure}
     \def\svgwidth{0.8\textwidth}
\begingroup%
  \makeatletter%
  \providecommand\color[2][]{%
    \errmessage{(Inkscape) Color is used for the text in Inkscape, but the package 'color.sty' is not loaded}%
    \renewcommand\color[2][]{}%
  }%
  \providecommand\transparent[1]{%
    \errmessage{(Inkscape) Transparency is used (non-zero) for the text in Inkscape, but the package 'transparent.sty' is not loaded}%
    \renewcommand\transparent[1]{}%
  }%
  \providecommand\rotatebox[2]{#2}%
  \newcommand*\fsize{\dimexpr\f@size pt\relax}%
  \newcommand*\lineheight[1]{\fontsize{\fsize}{#1\fsize}\selectfont}%
  \ifx\svgwidth\undefined%
    \setlength{\unitlength}{85.03937008bp}%
    \ifx\svgscale\undefined%
      \relax%
    \else%
      \setlength{\unitlength}{\unitlength * \real{\svgscale}}%
    \fi%
  \else%
    \setlength{\unitlength}{\svgwidth}%
  \fi%
  \global\let\svgwidth\undefined%
  \global\let\svgscale\undefined%
  \makeatother%
  \begin{picture}(1,0.33333333)%
    \lineheight{1}%
    \setlength\tabcolsep{0pt}%
    \put(0,0){\includegraphics[width=\unitlength,page=1]{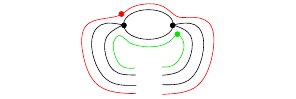}}%
    \put(0.49023366,0.06144486){\makebox(0,0)[lt]{\lineheight{1.25}\smash{\begin{tabular}[t]{l}$\dots$\end{tabular}}}}%
    \put(0.39985635,0.31132726){\makebox(0,0)[lt]{\lineheight{1.25}\smash{\begin{tabular}[t]{l}$b_0'$\end{tabular}}}}%
    \put(0.43658763,0.24064943){\makebox(0,0)[lt]{\lineheight{1.25}\smash{\begin{tabular}[t]{l}$b_0$\end{tabular}}}}%
    \put(0.54045272,0.24064943){\makebox(0,0)[lt]{\lineheight{1.25}\smash{\begin{tabular}[t]{l}$b_1$\end{tabular}}}}%
    \put(0.58632749,0.17284309){\makebox(0,0)[lt]{\lineheight{1.25}\smash{\begin{tabular}[t]{l}$b_1'$\end{tabular}}}}%
  \end{picture}%
\endgroup%

     \caption{Regular bad 1-simplex with $\nu=1$.}\label{fig:boundaries}
     \end{figure}

   \smallskip
    
    \para{Case 2: $\nu=2$} The fact that $b_0,b_1$ lie in different
    components implies that
    $$
    g(S\setminus\hat f (\sigma )) \geq g-p'\geq g-p
    $$
   as cutting along the first arc has no effect on the genus.  Hence induction here gives that $\mathcal D^\mu (S\setminus\hat f (\sigma ), b_0',b_1')$ is at least  $(\frac{2(g-p)-4}{3})$-connected.
   Now for all $p\geq 1$,
   $$k-p\leq \frac{2g-4}{3} - p = \frac{2g-3p-4}{3}  \leq \frac{2(g-p)-4}{3}$$
yielding the desired connectivity. 

\medskip

We will use the map $F$ of \eqref{equ:F} to modify $\hat f$ in the star $\f{St}(\sigma)$. 
For this purpose, note that as simplicial subcomplexes of $D^{k+1}$, 
    \begin{align*}
      \f{St}(\sigma) &= \sigma * \f{Lk}(\sigma), \quad \\
      \del \f{St}(\sigma) &= \del \sigma * \f{Lk}(\sigma).
    \end{align*}
 In particular, we get an identification $\del(\del \sigma * D^{k - p + 1}) \cong \del \f{St}(\sigma)$ for $D^{k-p+1}$ the simplicial disk that is the source of the map $F$ above.

  We replace $\hat f\, \vert\, _{\f{St}(\sigma)}$ by the unique simplicial map
    \[
      \hat f * F : \del \sigma * D^{k - p + 1}
      \to
      \mathcal B^\nu(S, b_0, b_1). 
    \]
    It remains to show that this has improved the situation.
    Indeed, suppose that  $\tau=\tau_0*\tau_1$ is a regular bad simplex in $\del \sigma * D^{k - p + 1}$. By construction, $\tau_1$ has image in
    $\mathcal D^\mu(S \setminus\hat f(\sigma), b_0', b_1') \subset    \mathcal D^\nu (S, b_0, b_1)$, so the ordering of the arcs of $\tau$ at $b_0$ and $b_1$ starts with the arcs of $\tau_1$, all in anti-clockwise order. Hence, if $\tau$ is regular bad, we must have $\tau=\tau_0$ is a strict face of $\s$. In particular, no new regular bad simplices have been added. As the simplex $\s$ has been removed, we have thus reduced the total  number of regular bad simplices in the disk. 
   Repeating this procedure,
        we will after finitely many stages
        remove every regular bad simplex,
        thus making the dashed arrow exist, which proves the result. 
      \end{proof}
      
\begin{remark}\label{rem:optimal1}
  The connectivity estimate above can be shown to be optimal
  in certain low-genus examples, corresponding to known computations
  of the unstable homology of mapping class groups.
  Indeed, $\mathcal D^2(S_{1,r})$ is disconnected.
  To see this, consider the spectral sequence associated
  to the action of the mapping class group $\Gamma (S_{1,r})$ on
  the simplicial complex $\mathcal D^2(S_{1,r})$.
  This is the spectral sequence arising from the vertical filtration
  of the double complex $\Z\mathcal D^2(S_{1,r})_\bullet\otimes_{\Gamma (S_{1,r})}
  F_\bullet$,
  where $F_\bullet\to\Z$ is a free resolution
  of the trivial $\Gamma (S_{1,r})$-module.
  By a standard argument using Shapiro's lemma (see e.g.~\cite[Thm
  5.1]{HatWah10} or \cite[Sec 1]{HatVog17}), one finds that
  the first page of this spectral sequence is given by
  $$
  E^1_{p,q} \cong
  \begin{cases}
    \widetilde H_q(\Gamma (S_{1,r})) &\text{if } p = -1,\\
    \widetilde H_q(\Gamma (S_{1,r-1})) &\text{if } p = 0,\\
    \widetilde H_q(\Gamma (S_{0,r})) &\text{if } p = 1,\\
    \widetilde H_q(\Gamma (S_{0,r-1})) &\text{if } p = 2,\\
    0 &\text{otherwise.}
  \end{cases}
  $$
  Assume for contradiction that $\mathcal D^2(S_{1,r})$ is connected.
  Then an analysis of the horizontal filtration of the double complex
  $\Z\mathcal D^2(S_{1,r})_\bullet\otimes_{\Gamma (S_{1,r})}F_\bullet$
  shows that $E^\infty_{p,q} = 0$ for $p+q\leq 0$,
  so the differential $d^1\colon H_1(\Gamma (S_{1,r-1}))\to H_1(\Gamma (S_{1,r}))$
  must be surjective. This contradicts the fact that
  $H_1(\Gamma (S_{1,s}))\cong\Z^s$ for $s\geq 1$ (see \cite[Thm
  5.1]{korkmaz}).
  Hence it is not true that $\mathcal D^\nu$ is
  $\left(\frac{2g+\nu -4} 3\right)$-connected when $\nu = 2$. 
  
  Similarly, one finds that $H_1(\mathcal D^1(S_{3,r}))\neq 0$
  by considering the spectral sequence associated to the action
  of $\Gamma (S_{3,r})$ on $\mathcal D^1(S_{2,r+1})$ and noting that
  the differential $d^1\colon H_1(\Gamma (S_{2,r+1}))\to H_1(\Gamma (S_{3,r}))$
  cannot be injective since the source identifies with $\Z / 10\Z $
  and the target is zero (see \cite[Thm 5.1]{korkmaz}).
  Thus $\mathcal D^\nu$ fails to be
  $\left(\frac{2g+\nu -4} 3\right)$-connected when $\nu = 1$ also.

  Note that these low dimensional computations also show that the
  first and last ranges in
  Theorem~\ref{thmintro:stab1} cannot be improved by a constant. 
%
\end{remark}

\section{The monoidal category of bidecorated surfaces}
\label{sec:category}
In this section, we describe a monoidal groupoid $(\M ,\hash )$ of surfaces
decorated by two intervals in their boundary, where the monoidal
structure glues the intervals in pairs. 
We show that this groupoid is a module over the braided monoidal
groupoid $\braidGrpd$ of braid groups, giving, on classifying spaces,
the structure of an $E_1$-module over an $E_2$-algebra in the sense of \cite{krannich19}.

\subsection{Bidecorated surfaces and the monoidal structure}\label{sec:bidecorated}

The groupoid $\M$ has objects {\em bidecorated surfaces}, that are, informally, surfaces with two intervals marked in their boundary. To give a precise definition of the objects that is convenient for the
monoidal structure, we start by 
 constructing a special sequence of bidecorated
 surfaces $X_n$, built out of disks, and defined inductively.
 
Let $X_1 = D^2\subset\C$ denote the unit disk in the complex plane,
and define the embeddings $\iota_{1}^0,\iota_{1}^1\colon I\to X_1$  by
$$
\iota_{1}^0(t) = e^{i(\pi /4 + t\pi /2)}\quad\text{and}\quad
\iota_{1}^1(t) = e^{i(5\pi /4 + t\pi /2)}.
$$
We denote by $\overline{\iota_{1}^i}\colon I\to X_1$  the reversed map
$t\mapsto \iota_{1}^i(1-t)$ for $i=0,1$.

Recursively, suppose we have defined $(X_m,\iota^0_{m},\iota_{m}^{1})$ for some $m\geq 1$.
We construct $X_{m+1}$  from $X_m$ by gluing an additional disk along
two half intervals, with new markings $\iota_{m+1}^0,\iota_{m+1}^1$
 coming from the first half of the markings of $X_m$ and
the second half of the markings of the attached disk:

$$  X_{m+1}
\defeq
     \frac{
      X_m \sqcup X_1
      }{
      \substack{
      \iota^i_m(t) \sim \overline{\iota^i_1}(t), \ t\in{[\rfrac12,1]},
    }
    }
\ \ \ \textrm{with}\ \ \ 
\iota_{m+1}^i(t)
=
\begin{cases}
  \iota_{m}^i(t),  & \text{if } t\leq 1/2,\\
  \iota_{1}^i(t),  & \text{else.}
\end{cases}
   $$
for $i=0,1$.
 Note that the marked intervals in the boundary of $X_m$ might live in different boundary components (in fact this will happen every other time). Figure~\ref{fig:glueD} shows what happens when a disk is glued to a surface in the above described manner, in each of these two possible cases.

\begin{figure}
  \def\svgwidth{0.8\textwidth}
\begingroup%
  \makeatletter%
  \providecommand\color[2][]{%
    \errmessage{(Inkscape) Color is used for the text in Inkscape, but the package 'color.sty' is not loaded}%
    \renewcommand\color[2][]{}%
  }%
  \providecommand\transparent[1]{%
    \errmessage{(Inkscape) Transparency is used (non-zero) for the text in Inkscape, but the package 'transparent.sty' is not loaded}%
    \renewcommand\transparent[1]{}%
  }%
  \providecommand\rotatebox[2]{#2}%
  \newcommand*\fsize{\dimexpr\f@size pt\relax}%
  \newcommand*\lineheight[1]{\fontsize{\fsize}{#1\fsize}\selectfont}%
  \ifx\svgwidth\undefined%
    \setlength{\unitlength}{85.03937008bp}%
    \ifx\svgscale\undefined%
      \relax%
    \else%
      \setlength{\unitlength}{\unitlength * \real{\svgscale}}%
    \fi%
  \else%
    \setlength{\unitlength}{\svgwidth}%
  \fi%
  \global\let\svgwidth\undefined%
  \global\let\svgscale\undefined%
  \makeatother%
  \begin{picture}(1,0.30666666)%
    \lineheight{1}%
    \setlength\tabcolsep{0pt}%
    \put(0,0){\includegraphics[width=\unitlength,page=1]{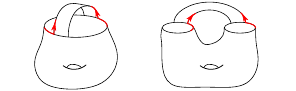}}%
    \put(0.16356215,0.11068636){\makebox(0,0)[lt]{\lineheight{1.25}\smash{\begin{tabular}[t]{l}$S$\end{tabular}}}}%
    \put(0.21903087,0.26593675){\makebox(0,0)[lt]{\lineheight{1.25}\smash{\begin{tabular}[t]{l}$D^2$\end{tabular}}}}%
    \put(0.64357084,0.11068795){\makebox(0,0)[lt]{\lineheight{1.25}\smash{\begin{tabular}[t]{l}$S$\end{tabular}}}}%
    \put(0.62661748,0.25895329){\makebox(0,0)[lt]{\lineheight{1.25}\smash{\begin{tabular}[t]{l}$D^2$\end{tabular}}}}%
  \end{picture}%
\endgroup%

  \caption{Gluing a disk $X_1 = D^2$
    to a bidecorated surface $S$ }\label{fig:glueD}
\end{figure}
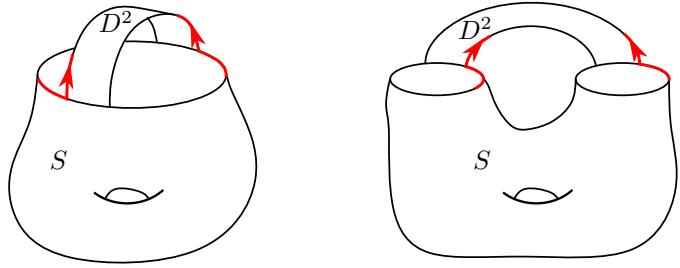

\begin{lemma}\label{lem:surface-type}
  Let $m \geq 1$.
  Then $X_m\cong S_{g,r}$ is a surface of genus $g$ with $r$ boundary components, where 
  \[
   (g,r)   =
    \begin{cases}
      (\frac{m}{2} - 1, 2), & \text{if $m$ is even},
      \\
      (\frac{m - 1}{ 2}, 1), & \text{if $m$ is odd}.
    \end{cases}
  \]
\end{lemma}

\begin{proof}
  Note first that $X_m$ is a connected surface for each $m$, since $X_1$ is a disk and $X_m$ is obtained from $X_1$ by successively adding disks (or strips), attached along two disjoint intervals in the boundary. 
For the same reason, we get that the Euler characteristic of $X_m$ is   
  $$
  \chi (X_{m+1}) = \chi (X_m) - 1  = \dots =  1 - m. 
  $$
By the classification of surfaces, we are left to compute the number of boundary components of $X_m$. 
For this, observing Figure~\ref{fig:glueD}, we notice that if we glue a disk along two intervals of $S$ that lie in the same boundary component, the new marked intervals given by the above procedure will give new intervals in different boundary components and vice versa, and no boundary component without marked intervals are ever created. It follows that the number of boundary components of $X_m$ alternates between $1$ and $2$. The result follows.  
\end{proof}

We are now ready to define the objects of the groupoid $\M$. We will use the boundary of the above defined surfaces $X_m$ to parametrize the boundary components of the surfaces that contain the marked intervals, to allow us to work with parametrized boundary components instead of parametrized arcs, in order to simplify some definitions. 
\begin{definition}\label{def:bidec-surfaces}
  A \emph{bidecorated surface} is a tuple $(S,m,\varphi)$
  where $S$ is a surface,
  $m\ge 1$ is an integer, and 
  $$
  \varphi\colon \del X_m\sqcup (\sqcup_kS^1) \xrightarrow\sim \del S
  $$
  is a homeomorphism, giving a parametrization of the boundary of $S$. 
  We think of $(S,m,\varphi)$ as a surface with two parametrized arcs
  $$
  I_0 \defeq \varphi\circ\iota_{m}^0\quad\text{and}\quad
  I_1 \defeq \varphi\circ\iota_{m}^1
  $$
  in its boundary, and $k$ additional parametrized boundaries. The
  surface $S$ may also have punctures. 
 \end{definition}
  The monoidal groupoid 
  $(\M, \hash,U)$ has objects the bidecorated surfaces together with
  a formal unit $U$.
  There are no morphisms between two bidecorated surfaces
  $(S,m ,\varphi)$ and $(S',m',\varphi ')$ unless $S$ and $S'$
  are homeomorphic and $m=m'$,
  in which case we define the set of morphisms to be all the mapping
  classes of homeomorphisms
  that preserve the boundary parametrizations
  \begin{multline*}
    \Hom_\M((S,m,\varphi), (S',m,\varphi '))
    \defeq
        \pi_0\Homeo_{\del}(S,S') = \pi_0\lbrace f\in\Homeo (S,S')\mid f\circ\varphi = \varphi '\rbrace, 
      \end{multline*}
      where $\Homeo (S,S' )$ is endowed with the compact-open topology,
      and $\Homeo_{\del}(S,S')$  with  the subspace topology.  
 The only morphism involving the unit $U$
 is the identity $\text{id}_U$.

 \begin{remark}
Our definition of the morphisms in the category $\M$ is such that
punctures in a surfaces $S$ can be permuted by automorphisms of $S$ in
$\M$. Our argument works just as well with labeled punctures, that are
not permutable by homeomorphims, or both labeled and unlabeled
punctures, just like we could also have additional boundary components
that are  only marked up to a permutation. The only changes this would
cause to the
argument would be that it would make the notations and conventions more
cumbersome.  
   \end{remark}
 
  The monoidal structure $\hash$ is defined as follows. 
  The object $U$ is by definition a unit for $\hash$.
  For the remaining objects, the monoidal product $\hash$
  is defined by
  \begin{align*}
    (S,m,\varphi)\hash (S',m',\varphi ')
\defeq
      \(\frac{
      S \sqcup S'
      }{
      \substack{
      I_i(t) \sim \overline{I_i'}(t), t\in {[\rfrac12,1]},
    }
    },m+m',\varphi\hash\varphi ' \),
  \end{align*}
 where $i=0,1$, and where 
  $$
  \varphi\hash\varphi '
  \colon \del X_{m+m'}\sqcup (\sqcup_{k+k'}S^1) \ \inc\ \del (S\hash S'),
  $$
  is obtained using the canonical identification $\del X_{m+m'}\cong
  (\del X_n\backslash \iota_m(\frac{1}{2},1))\cup (\del
  X_{m'}\backslash \iota_{m'}(0,\frac{1}{2}))$. 
  On morphisms, the monoidal product is given by juxtaposition.

  \medskip
  
The monoidal category $\M$ has the following  {\em injectivity
  property} with respect to gluing a disk, that will be useful in the
proof of our stability result. 

\begin{proposition}\label{prop:Mmonoid}
  For any object $S=(S,m,\varphi)$ of $\M$, and any $p\ge 0$, 
  the map
  \[
    \Aut_\M(S)
    \tox{\hash D^{\hash{p+1}}}
    \Aut_\M(S \hash D^{\hash {p+1}})
  \]
  is injective, where $D=(X_1,1,\id)$ is our chosen disk.  
\end{proposition}

\begin{proof} Recall that the underlying surface of $D^{\hash p+1}$ is the surface  $X_{p+1}$ defined above. 
  Picking a smooth representative of the underlying surface of $S\hash
  X_{p+1}$, with $S$ a smooth subsurface in its interior, we can model
  the map in the statement using the description of the mapping class
  group of surfaces in terms of isotopy classes of diffeomorphisms
  rather than homeomorphisms. (See e.g., \cite[Thm 1.2]{boldsen2009}
  for a detailed account of the classical isomorphism
  $\pi_0\Homeo_{\del}(S)\cong \pi_0\Diff_\del(S)$
  when $S$ is compact.)
  Now the result follows by essentially the same argument as 
 the case of  attaching of surface along a single arc instead of two, as treated in \cite[Prop 5.18]{RWW17}, using the fibration
 \begin{align*}
  \Diff(S\hash X_{p+1} \ \textrm{rel}\ \del S\cup X_{p+1}) &\rar
\Diff(S\hash X_{p+1} \ \textrm{rel}\ \del (S\hash X_{p+1}))\\
   &\rar \text{Emb}((X_{p+1},I_0|_{[\frac{1}{2},1]}\cup I_1|_{[\frac{1}{2},1]}),(S\hash X_{p+1},I_0|_{[\frac{1}{2},1]}\cup I_1|_{[\frac{1}{2},1]}))
 \end{align*}
 where the fiber identifies with $\Diff(S\ \textrm{rel}\ \del_0S)$ and where we note that $I_0|_{[\frac{1}{2},1]}\cup I_1|_{[\frac{1}{2},1]}=\del X_{p+1}\cap \del(S\hash X_{p+1})$.
 Injectivity of the first map on $\pi_0$ follows if we can show that the base is simply-connected. 
 In fact the base can be shown inductively to have contractible components, using that $X_{p+1}$ is built inductively by attaching disks along two intervals, or homotopically attaching arcs, and using the contractibility of the components of embeddings of arcs in a surface, as proved in \cite[Thm 5]{Gramain}. 
  \end{proof}

  \subsection{Braided action}\label{sec:braidedaction}
  We want to apply the homological stability machine of \cite{krannich19} to stabilization
  in $\M$ with the bidecorated disk
  $$
  D \defeq (X_1,1,\text{id}).
  $$
  For this, we need that
  the classifying space of $\M$ is an $E_1$--module over an $E_2$--algebra.
  This will follow if we can show on the categorical level that
  $\M$ admits an appropriate action of a braided monoidal groupoid.
  We will build such an action in this section,
  using as braided monoidal groupoid the groupoid of braid groups.
  In constrast with most classical examples of homological stability,
  we will show in Section~\ref{sec:notbraided} that
  this action of the braid groupoid does not come from a braided structure on $\M$,
  or the full monoidal subcategory generated by $D$.
  It is instead constructed using a {\em
    Yang--Baxter element} in $\M$, associated to a braid subgroup of
  the mapping class group of $X_m$, that we will describe now.

\smallskip
  
Write
 $$
 D^{\hash m}
 =
 D_1 \hash \ldots \hash (D_i \hash D_{i+1}) \hash \ldots \hash D_m,
 $$
 where we use subscript to enumerate the disks, and where the
 underlying surface is $X_m$.
 We let $a_i$ denote the isotopy class of a curve in the interior $D_i
  \hash D_{i+1}\cong S^1\x I$
  that is parallel to its boundary components,
  as shown in Figure~\ref{fig:rhoiai}.
\begin{figure}
    \def\svgwidth{0.8\textwidth}
\begingroup%
  \makeatletter%
  \providecommand\color[2][]{%
    \errmessage{(Inkscape) Color is used for the text in Inkscape, but the package 'color.sty' is not loaded}%
    \renewcommand\color[2][]{}%
  }%
  \providecommand\transparent[1]{%
    \errmessage{(Inkscape) Transparency is used (non-zero) for the text in Inkscape, but the package 'transparent.sty' is not loaded}%
    \renewcommand\transparent[1]{}%
  }%
  \providecommand\rotatebox[2]{#2}%
  \newcommand*\fsize{\dimexpr\f@size pt\relax}%
  \newcommand*\lineheight[1]{\fontsize{\fsize}{#1\fsize}\selectfont}%
  \ifx\svgwidth\undefined%
    \setlength{\unitlength}{85.03937008bp}%
    \ifx\svgscale\undefined%
      \relax%
    \else%
      \setlength{\unitlength}{\unitlength * \real{\svgscale}}%
    \fi%
  \else%
    \setlength{\unitlength}{\svgwidth}%
  \fi%
  \global\let\svgwidth\undefined%
  \global\let\svgscale\undefined%
  \makeatother%
  \begin{picture}(1,0.46666667)%
    \lineheight{1}%
    \setlength\tabcolsep{0pt}%
    \put(0,0){\includegraphics[width=\unitlength,page=1]{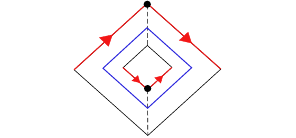}}%
    \put(0.49581401,0.18884676){\makebox(0,0)[lt]{\lineheight{1.25}\smash{\begin{tabular}[t]{l}$b_1$\end{tabular}}}}%
    \put(0.51640974,0.45716485){\makebox(0,0)[lt]{\lineheight{1.25}\smash{\begin{tabular}[t]{l}$b_0$\end{tabular}}}}%
    \put(0.39131347,0.12541836){\makebox(0,0)[lt]{\lineheight{1.25}\smash{\begin{tabular}[t]{l}$a_i$\end{tabular}}}}%
  \end{picture}%
\endgroup%

  \caption{The curve $a_i$ in $D_i\hash D_{i+1}$}\label{fig:rhoiai}
\end{figure}
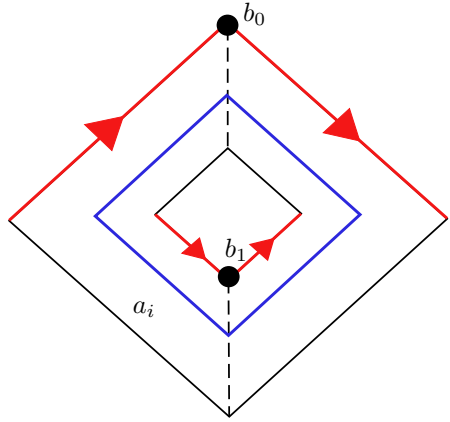

\begin{lemma}\label{lem:chain}
  The curves $a_1,\dots,a_{m-1}$ form a {\em chain} in $D^{\hash m}$, i.e. $a_i$ and $a_{i+1}$ have intersection number 1 for each $i$, and $a_i\cap a_j=\emptyset$ if $|i-j|>1$. 
\end{lemma}

\begin{proof}
  The curve $a_i$ lives in the disks $D_i$ and $D_{i+1}$,
  so it can only intersect $a_{i-1}$ and $a_{i+1}$ non-trivially,
  and hence it suffices to consider
  the subsurface of $D^{\hash m}$ corresponding to
  $D_i\hash D_{i+1}\hash D_{i+2}$.
  Here the claim can be checked by hand, see Figure~\ref{fig:crossing}.
\end{proof}
\begin{figure}
  \def\svgwidth{0.65\textwidth}
\begingroup%
  \makeatletter%
  \providecommand\color[2][]{%
    \errmessage{(Inkscape) Color is used for the text in Inkscape, but the package 'color.sty' is not loaded}%
    \renewcommand\color[2][]{}%
  }%
  \providecommand\transparent[1]{%
    \errmessage{(Inkscape) Transparency is used (non-zero) for the text in Inkscape, but the package 'transparent.sty' is not loaded}%
    \renewcommand\transparent[1]{}%
  }%
  \providecommand\rotatebox[2]{#2}%
  \newcommand*\fsize{\dimexpr\f@size pt\relax}%
  \newcommand*\lineheight[1]{\fontsize{\fsize}{#1\fsize}\selectfont}%
  \ifx\svgwidth\undefined%
    \setlength{\unitlength}{85.03937008bp}%
    \ifx\svgscale\undefined%
      \relax%
    \else%
      \setlength{\unitlength}{\unitlength * \real{\svgscale}}%
    \fi%
  \else%
    \setlength{\unitlength}{\svgwidth}%
  \fi%
  \global\let\svgwidth\undefined%
  \global\let\svgscale\undefined%
  \makeatother%
  \begin{picture}(1,0.66666667)%
    \lineheight{1}%
    \setlength\tabcolsep{0pt}%
    \put(0,0){\includegraphics[width=\unitlength,page=1]{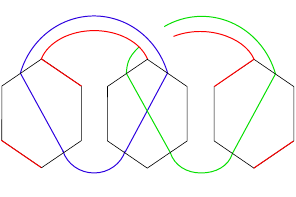}}%
    \put(0.02291926,0.26440439){\makebox(0,0)[lt]{\lineheight{1.25}\smash{\begin{tabular}[t]{l}$D_i$\end{tabular}}}}%
    \put(0.37865254,0.26440439){\makebox(0,0)[lt]{\lineheight{1.25}\smash{\begin{tabular}[t]{l}$D_{i+1}$\end{tabular}}}}%
    \put(0.74533518,0.26440439){\makebox(0,0)[lt]{\lineheight{1.25}\smash{\begin{tabular}[t]{l}$D_{i+2}$\end{tabular}}}}%
    \put(0,0){\includegraphics[width=\unitlength,page=2]{intersection-of-ai-curves.pdf}}%
  \end{picture}%
\endgroup%

  \caption{Intersection of $a_i$ (blue) and $a_{i+1}$ (green) in the
    underlying surface of $D_i\hash D_{i+1}\hash D_{i+2}$.}
  \label{fig:crossing}
\end{figure}

Let $T_i\in\text{Aut}_{\M}(D^{\hash m})$
denote the Dehn twist along the curve $a_i$ in $D^{\hash m}$.
A classical fact states that the Dehn twists along a chain of embedded
curves satisfy the braid relations (see e.g.~\cite[3.9 and 3.11]{MCG-primer}):
\begin{equation}\label{equ:braidrel}
  \begin{aligned}
  T_{i}T_{i+1}T_{i} & = T_{i+1}T_{i}T_{i+1} &&\quad\text{for all } i, \\
  T_iT_j & = T_jT_i  &&\quad\text{if } |i - j| > 1, 
\end{aligned}
\end{equation}
Note that the same relations are satisfied by the inverse twists $T_i^{-1}$,
that will turn out to be more convenient for us.
Also, adding a disk to the right or left of $D^{\hash m}$ gives the relations
$$
T_i\hash\text{id}_D = T_i\quad\text{and}\quad\text{id}_D\hash T_i = T_{i+1}
$$
in $\Aut_{\M}(D^{\hash m+1})$.
In particular~\eqref{equ:braidrel} includes the relation
$$
(T_1^{-1}\hash\text{id}_D)
(\text{id}_D\hash T_1^{-1})
(T_1^{-1}\hash\text{id}_D)
=
(\text{id}_D\hash T_1^{-1})
(T_1^{-1}\hash\text{id}_D)
(\text{id}_D\hash T_1^{-1})
$$
in $\Aut_\M (D^{\hash 3})$,
so in other words,
the inverse Dehn twist $T_1^{-1}\in \Aut_\M (D{\hash} D)$ is a Yang--Baxter operator
in the sense of Section~\ref{sec:YB}.

Recall from the introduction that $\BB$ denotes the groupoid of braid
groups, with objects the natural numbers $\lbrace 0,1,2,\dots\rbrace$,
automorphisms of $n$ the
braid group $B_n$, and no other non-trivial morphisms. 
In Section~\ref{sec:YB} we show that, being a Yang-Baxter operator, $T_1^{-1}$ yields
a strong monoidal functor
$$
\Phi = \Phi_{D,T_1^{-1}}\colon (\BB ,\oplus )\to (\M ,\hash ),
$$
uniquely determined up to monoidal natural isomorphism by the fact
that $\Phi (1) = D$ and, for the standard generator $\sigma_1\in B_2= \text{Aut}_\BB (1)$,
$\Phi (\sigma_1) = T_1^{-1}$.

\smallskip
Such a functor $\Phi$  endows $\M$ with the structure
of an $E_1$-module over $\BB$ via the associated functor
\begin{align*}
 \alpha=( - \hash \Phi(-)) \colon \  \M\times \BB &\rar \M,  
\end{align*}
given on objects by $\alpha(S,n)= S\hash\Phi (n) = S\hash D^{\hash
  n}$, and likewise for morphisms. 
On classifying spaces, this yields exactly the kind of input needed in Krannich's homological stability
framework, see \cite[Lem 7.2]{krannich19}.

\begin{remark}
  For each $m$, the restriction of the functor $\Phi: \BB\to \M$ to
  $B_m = \text{Aut}_\BB(m)$ maps the standard generator $\sigma_i$
  to the inverse Dehn twist $T_i^{-1}\in \Aut_\M (D^{\hash
    m})=\pi_0\textrm{Homeo}_\partial(X_m)$.
  By Birman--Hilden theory~\cite{BH1,BH2}
  the homomorphisms $\Phi\vert_{B_m}\colon B_m\to\text{Aut}_\M (D^{\hash m})$
  are actually injective. 
\end{remark}

\section{Homological stability}
\label{sec:stab}

Generalizing the main result of \cite{RWW17},
Krannich associates to an $\E_1$-module $\mathcal M$
over an $\E_2$-algebra $\mathcal X$
with a chosen stabilizing object $X\in\mathcal X$,
a {\em space of destabilizations} at every $A\in\mathcal M$,
whose high connectivity implies  homological stability at $A$
when stabilizing by $X$.
We are interested in the case
where $\mathcal M=B\M$ is the classifying
space of $\M$ and $\mathcal X=B\BB$, acting on $B\M$
via the map $\alpha\colon \M\times \BB\to \M$ defined in
Section~\ref{sec:braidedaction}. 
We will pick $A=S\in \M$ to be some surface, with $X=1\in \BB$
 modelling stabilization with the disk as $\alpha(-,X)=-\hash D$ is the sum with the  bidecorated disk $D=(X_1,1,\id)$ of
Section~\ref{sec:bidecorated}. 

Generally,
the space of destabilizations is a semi-simplicial space,
but in settings such as ours, it
is actually levelwise homotopy discrete. Indeed, by  \cite[Lem
7.6]{krannich19}), 
when the structure of $\E_1$-module
over an $\E_2$-algebra is induced by an action of a braided monoidal
category on a groupoid, and under the  injectivity condition given in 
Proposition~\ref{prop:Mmonoid}, the space of destabilizations 
is equivalent to the following semi-simplicial set,
defined just as in \cite{RWW17} in 
the case of a braided monoidal groupoid acting on itself.

\begin{definition}(\cite[Def 7.5]{krannich19})\label{def:WS}
  Let $(\mathcal M,\oplus)$ be a right module
  over a braided monoidal groupoid $(\mathcal X,\oplus,b)$,
  where we denote also by $\oplus$ the module action.
  Let $A$ and $X$ be objects of $\mathcal M$ and $\mathcal X$ respectively.
  The {\em space of destabilizations} $W_n(A,X)_\bullet$ is the semi-simplicial
  set with set of $p$-simplices 
\begin{align*}
  W_n(A,X)_p 
                 =& \ \{(B,f)\ |\ B\in \textrm{Ob}(\mathcal M) \ \textrm{and}\ f\colon B\oplus X^{\oplus p+1}\to A\oplus X^{\oplus n} \ \textrm{in}\ \mathcal M\}/_\sim
\end{align*}
where $(B,f)\sim (B',f')$ if there exists an isomorphism $g\colon B\to B'$ in $\mathcal C$ satisfying that $f=f'\circ (g\oplus \id_{X^{\oplus p+1}})$. 
The face map $d_i\colon W_n(A,X)_p\to W_n(A,X)_{p-1}$ is defined by $d_i[B,f]=[B\oplus X, d_if]$ for  
$$d_i f\colon B\oplus X\oplus X^{p} \xrightarrow{\id_B\oplus b_{X^{\oplus i},X}^{-1}\oplus \id_{X^{\oplus p-i}}} B\oplus X^{\oplus i}\oplus X\oplus X^{\oplus p-i}\xrightarrow{\ f\ } A\oplus X^{\oplus n},$$
for $b_{X^{\oplus i},X}^{-1}\colon X\oplus X^{\oplus i}\to  X^{\oplus i}\oplus X$ coming from the braiding in $\mathcal X$. 
\end{definition}

\subsection{Disk destabilizations and disordered arcs}\label{sec:disk}

Given a bidecorated orientable\footnote{The definition of the
  disordered arc complex naturally extend to non-orientable
  bidecorated surfaces, ordering the arcs according to the
  orientations of $I_0$ and $I_1$, but we will only consider orientable surfaces
  here} surface $S = (S,m,\varphi )$, with $I_0,I_1$ compatibly oriented, let $\mathcal D(S) = \mathcal D^\nu(S,b_0,b_1)$ 
denote the disordered arc complex of $S$ as in Section~\ref{sec:high-cnt}, where 
$$
b_0 = I_0(1/2)\quad\text{and}\quad b_1 = I_1(1/2)
$$
are the midpoints of the marked intervals, and $\nu=1$ if $I_0$ and
$I_1$ lie on the same boundary component and $\nu=2$ otherwise. 
The vertices of a simplex in $\mathcal D(S)$ are canonically
ordered by the anti-clockwise ordering at $b_0$ 
(or equivalently at $b_1$).
Hence we can associate to this simplicial complex a semi-simplicial set
that we denote $\mathcal D(S)_\bullet$,
with same set of $p$-simplices and whose $i$th face map is given
by forgetting the $(i+1)$st arc with respect to that  ordering.
As $\mathcal D(S)$ and  $\mathcal D(S)_\bullet$ have homeomorphic realizations,
they have the same connectivity.

\medskip

Write  $W_n(S,D)_\bullet$ for the space of destabilization of
Definition~\ref{def:WS} associated to the module 
 $\mathcal M=\M$ over the $E_2$--algebra $\mathcal X=\BB$ acting on $\M$ as above,
with 
$X=1\in \BB$,
and $A=S=(S_{g,r}^s,m,\varphi)$ some bidecorated
orientable surface of small genus 
$g\ge 0$, with $r$ boundary components and $s$ punctures. 
The space $W_n(S,D)_\bullet$ is then  the space of destabilizations
of the stabilization map
$$
\Aut_\M (S\hash D^{\hash n-1}) \xrightarrow{\hash D} \Aut_\M (S\hash D^{\hash n})
$$
that attaches an additional disk to the surface along the two marked
intervals.

We want to identify $W_n(S,D)_\bullet$ with $\mathcal D(S\hash
D^{\hash n})_\bullet$. 
For this, we start by constructing  a particular disordered collection of arcs in $D^{\hash n}$.
Write again
$$
D^{\hash n} = D_1\hash \dots \hash D_i\hash\dots\hash
D_n, 
$$
and let $\rho_i$ denote the unique isotopy class of arc in the $i$th disk $D_i$
going from $b_0 = I_0(1/2)$ to $b_1 = I_1(1/2)$.
\begin{lemma}\label{lem:rhoord}
  The arcs $\rho_1,\dots,\rho_m$ are ordered anti-clockwise at both $b_0$ and $b_1$.
\end{lemma}

\begin{proof}
  It suffices to show that $\rho_i$ and $\rho_{i+1}$ are ordered anti-clockwise
  at $b_0$ and $b_1$ for each $i$.
  Thus we need only consider what happens in the subsurface $D_i\hash D_{i+1}$.
  The gluing being defined in exactly the same way at $I_0$ and $I_1$,
  the arcs are ordered in the same way at both endpoints,
  and the particular choice of gluing gives the anti-clockwise ordering,
  see  Figure~\ref{fig:rhoiorder}.
\end{proof}
\begin{figure}
  \def\svgwidth{0.8\textwidth}
\begingroup%
  \makeatletter%
  \providecommand\color[2][]{%
    \errmessage{(Inkscape) Color is used for the text in Inkscape, but the package 'color.sty' is not loaded}%
    \renewcommand\color[2][]{}%
  }%
  \providecommand\transparent[1]{%
    \errmessage{(Inkscape) Transparency is used (non-zero) for the text in Inkscape, but the package 'transparent.sty' is not loaded}%
    \renewcommand\transparent[1]{}%
  }%
  \providecommand\rotatebox[2]{#2}%
  \newcommand*\fsize{\dimexpr\f@size pt\relax}%
  \newcommand*\lineheight[1]{\fontsize{\fsize}{#1\fsize}\selectfont}%
  \ifx\svgwidth\undefined%
    \setlength{\unitlength}{85.03937008bp}%
    \ifx\svgscale\undefined%
      \relax%
    \else%
      \setlength{\unitlength}{\unitlength * \real{\svgscale}}%
    \fi%
  \else%
    \setlength{\unitlength}{\svgwidth}%
  \fi%
  \global\let\svgwidth\undefined%
  \global\let\svgscale\undefined%
  \makeatother%
  \begin{picture}(1,0.46666667)%
    \lineheight{1}%
    \setlength\tabcolsep{0pt}%
    \put(0,0){\includegraphics[width=\unitlength,page=1]{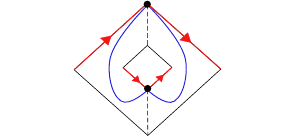}}%
    \put(0.49581401,0.18884676){\makebox(0,0)[lt]{\lineheight{1.25}\smash{\begin{tabular}[t]{l}$b_1$\end{tabular}}}}%
    \put(0.51640974,0.45716485){\makebox(0,0)[lt]{\lineheight{1.25}\smash{\begin{tabular}[t]{l}$b_0$\end{tabular}}}}%
    \put(0.32075787,0.21701872){\makebox(0,0)[lt]{\lineheight{1.25}\smash{\begin{tabular}[t]{l}$\rho_i$\end{tabular}}}}%
    \put(0.65198606,0.21701872){\makebox(0,0)[lt]{\lineheight{1.25}\smash{\begin{tabular}[t]{l}$\rho_{i+1}$\end{tabular}}}}%
  \end{picture}%
\endgroup%

  \caption{Ordering of the arcs $\rho_i$ at their endpoints}\label{fig:rhoiorder}
\end{figure}
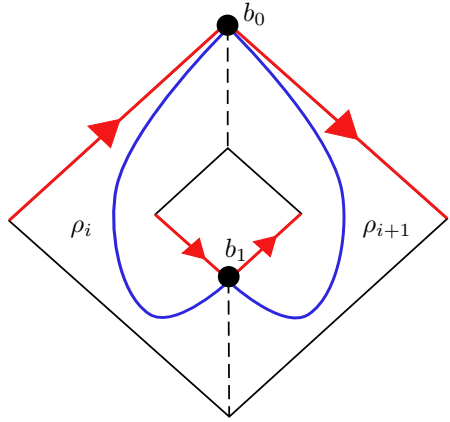

Recall from Section~\ref{sec:braidedaction} the Dehn twist $T_i$ along
the curve $a_i$ in $D_i\hash D_{i+1}$.
The union of the arcs $\rho_i$ in $D^{\hash m}$ define a deformation retract of the surface,
as each disk $D_i$ retracts onto the corresponding arc $\rho_i$, 
and we can understand the action of the twists $T_i$ on the surface
by considering their action on the arcs $\rho_i$. The action is given
by the following result, that will be needed to compare the face maps
in the semi-simplicial sets $W_n(S,D)_\bullet$ with $\mathcal D(S\hash
D^{\hash n})_\bullet$. 
\begin{lemma}\label{lem:Trho}
  The action of the Dehn twist $T_i$ along the curve $a_i$ on the homotopy classes of the arcs $\rho_i$, relative to their endpoints, is 
  $$
  T_i(\rho_j)
  =
  \begin{cases}
    \rho_i\overline{\rho_{i+1}}\rho_i &\text{if } j = i, \\
    \rho_{i} &\text{if } j = i+1, \\
    \rho_j &\text{else.}
  \end{cases}
  $$
 Equivalently, 
 $$T_i^{-1}(\rho_i)=\rho_{i+1} \ \ \ \textrm{and} \ \ \ T_i^{-1}(\rho_{i+1})=\rho_{i+1}\overline{\rho_i}\rho_{i+1}$$
 and $T_i^{-1}$ leaves the other $\rho_j$ invariant. 
\end{lemma}
\begin{proof}
  The Dehn twist $T_i$ can only affect $\rho_i$ and $\rho_{i+1}$ as the curve $a_i$ only intersects these two arcs, from which the last case in the statement follows.
  The computation for the arcs $\rho_i$ and $\rho_{i+1}$ is local to
  $D_i\hash D_{i+1}$, where, as shown in Figure~\ref{fig:twist},
  we have $T_i(\rho_i)\simeq\rho_i\overline{\rho_{i+1}}\rho_i$,
  giving the first case in the statement,
  and $T_i(\rho_{i+1})\simeq\rho_i$, giving the second case.
\end{proof}
\begin{figure}
  \def\svgwidth{0.8\textwidth}
\begingroup%
  \makeatletter%
  \providecommand\color[2][]{%
    \errmessage{(Inkscape) Color is used for the text in Inkscape, but the package 'color.sty' is not loaded}%
    \renewcommand\color[2][]{}%
  }%
  \providecommand\transparent[1]{%
    \errmessage{(Inkscape) Transparency is used (non-zero) for the text in Inkscape, but the package 'transparent.sty' is not loaded}%
    \renewcommand\transparent[1]{}%
  }%
  \providecommand\rotatebox[2]{#2}%
  \newcommand*\fsize{\dimexpr\f@size pt\relax}%
  \newcommand*\lineheight[1]{\fontsize{\fsize}{#1\fsize}\selectfont}%
  \ifx\svgwidth\undefined%
    \setlength{\unitlength}{85.03937008bp}%
    \ifx\svgscale\undefined%
      \relax%
    \else%
      \setlength{\unitlength}{\unitlength * \real{\svgscale}}%
    \fi%
  \else%
    \setlength{\unitlength}{\svgwidth}%
  \fi%
  \global\let\svgwidth\undefined%
  \global\let\svgscale\undefined%
  \makeatother%
  \begin{picture}(1,0.38773174)%
    \lineheight{1}%
    \setlength\tabcolsep{0pt}%
    \put(0,0){\includegraphics[width=\unitlength,page=1]{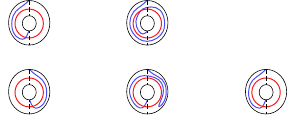}}%
    \put(0.27355328,0.33500277){\makebox(0,0)[lt]{\lineheight{1.25}\smash{\begin{tabular}[t]{l}$T_i$\end{tabular}}}}%
    \put(0.26849543,0.10059468){\makebox(0,0)[lt]{\lineheight{1.25}\smash{\begin{tabular}[t]{l}$T_i$\end{tabular}}}}%
    \put(0.69529006,0.06843478){\makebox(0,0)[lt]{\lineheight{1.25}\smash{\begin{tabular}[t]{l}$\simeq$\end{tabular}}}}%
    \put(0.03405323,0.22986291){\makebox(0,0)[lt]{\lineheight{1.25}\smash{\begin{tabular}[t]{l}$\color{blue} \rho_i$\end{tabular}}}}%
    \put(0.16320898,0.0211572){\makebox(0,0)[lt]{\lineheight{1.25}\smash{\begin{tabular}[t]{l}$\color{blue} \rho_{i+1}$\end{tabular}}}}%
    \put(0,0){\includegraphics[width=\unitlength,page=2]{action-of-t-on-rho.pdf}}%
  \end{picture}%
\endgroup%

  \caption{The action of the Dehn twist $T_i$ on the arcs
    $\rho_i$ (top) and $\rho_{i+1}$ (bottom)}\label{fig:twist}
\end{figure}
 
  \begin{proposition}\label{prop:Wiso}
    Let $S=(S,m,\varphi)$
    be an object of $\M$.
    There is an isomorphism of semi-simplicial sets
    $$W_n(S,D)_\bullet\cong \mathcal D^\nu(S\hash D^{\hash n})_\bullet$$
 where the marked points $b_0$ and $b_1$ are the midpoints of the
 intervals $I_0$ and $I_1$ in $S\hash D^{\hash n}$ and with 
 $\nu=\operatorname{parity}(m+n)$, that is $\nu=1$ if  $I_0$ and $I_1$ lie in the same boundary component of $S\hash D^{\hash n}$ and  $\nu=2$ otherwise. 
  \end{proposition}
\begin{proof}
  We first show that both $W_n(S,D)_p$
  and $\mathcal D^\nu(S\hash D^{\hash n})_p$ are isomorphic to
  $\Aut_\M(S\hash D^{\hash n})/\Aut_\M(S\hash D^{\hash n-p-1})$
  for every $p\ge 0$.
  This holds  by definition for the first semi-simplicial set.
  For $\mathcal D^\nu(S\hash D^{\hash n})_p$,
  it will follow from two facts:
  (1) the natural action of
  $$
  \Aut_\M (S\hash D^{\hash n}) = \pi_0\Homeo_\partial (S\hash D^{\hash
    n})
  $$
  on this set of $p$-simplices is transitive, and
  (2) the stabilizer of a $p$-simplex is isomorphic to
  $\Aut_\M(S\hash D^{\hash n-p-1})$.
  The first fact follows because the homeomorphism type
  of the complement $S\setminus \s$ of a collection of
  non-separating arcs $\s=\<a_0,\dots,a_p\>$ is determined
  by the orderings of the arcs at the endpoints as this
  determines the number of boundary components of the complement (see
  \cite[Lem 3.2]{harer85}),
  and the second from the fact that this complement is precisely diffeomorphic
  to $S\hash D^{\hash n-p-1}$ for any $p$-simplex in the disordered
  arc complex. Indeed, 
 this diffeomorphism type does not depend on the simplex by
 transitivity of the action, so it is
 enough to check the claim for any chosen
 simplex. Let $$\s_p=\<\rho_{n-p},\dots,\rho_n\>$$
 be the collection of arcs in $S\hash D^{\hash n}$ consisting of the
 cores $\rho_i$ of the last $p+1$ disks. Recall from Lemma~\ref{lem:rhoord}
 that this is a disordered simplex, once we note additionally that the arcs are
 also non-separating. Now Figure~\ref{fig:cutrho} shows
 that the operation of cutting along the core $\rho$ of a disk exactly undoes
 the gluing operation, which proves the claim in that case.
 \begin{figure}
   \def\svgwidth{0.8\textwidth}
\begingroup%
  \makeatletter%
  \providecommand\color[2][]{%
    \errmessage{(Inkscape) Color is used for the text in Inkscape, but the package 'color.sty' is not loaded}%
    \renewcommand\color[2][]{}%
  }%
  \providecommand\transparent[1]{%
    \errmessage{(Inkscape) Transparency is used (non-zero) for the text in Inkscape, but the package 'transparent.sty' is not loaded}%
    \renewcommand\transparent[1]{}%
  }%
  \providecommand\rotatebox[2]{#2}%
  \newcommand*\fsize{\dimexpr\f@size pt\relax}%
  \newcommand*\lineheight[1]{\fontsize{\fsize}{#1\fsize}\selectfont}%
  \ifx\svgwidth\undefined%
    \setlength{\unitlength}{85.03937008bp}%
    \ifx\svgscale\undefined%
      \relax%
    \else%
      \setlength{\unitlength}{\unitlength * \real{\svgscale}}%
    \fi%
  \else%
    \setlength{\unitlength}{\svgwidth}%
  \fi%
  \global\let\svgwidth\undefined%
  \global\let\svgscale\undefined%
  \makeatother%
  \begin{picture}(1,0.36666667)%
    \lineheight{1}%
    \setlength\tabcolsep{0pt}%
    \put(0,0){\includegraphics[width=\unitlength,page=1]{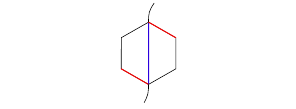}}%
    \put(0.42234271,0.17565319){\makebox(0,0)[lt]{\lineheight{1.25}\smash{\begin{tabular}[t]{l}$D$\end{tabular}}}}%
    \put(0.5148025,0.21073919){\makebox(0,0)[lt]{\lineheight{1.25}\smash{\begin{tabular}[t]{l}$\rho$\end{tabular}}}}%
    \put(0.3983782,0.32612871){\makebox(0,0)[lt]{\lineheight{1.25}\smash{\begin{tabular}[t]{l}$S$\end{tabular}}}}%
    \put(0.61583871,0.03166797){\makebox(0,0)[lt]{\lineheight{1.25}\smash{\begin{tabular}[t]{l}$S$\end{tabular}}}}%
    \put(0,0){\includegraphics[width=\unitlength,page=2]{cutting-along-disc.pdf}}%
  \end{picture}%
\endgroup%

   \caption{Cutting along the core of a disk}\label{fig:cutrho}
   \end{figure}

Note that the actions on both sets of simplices are  given by post-composition with mapping  classes, where we think here of an arc as an isotopy class of embedding. 
There is then a unique equivariant 
isomorphism
$\phi_p\colon W_n(S,D)_p\xrightarrow{\cong} \mathcal D^\nu(S\hash D^{\hash n})_p$
taking the $p$-simplex 
$$
f_p=(S\hash D^{\hash n-p-1}, \,\id_{S\hash D^{\hash n}})  
$$
of $W_n(S,D)$ to the  $p$-simplex 
$\s_p=\<\rho_{n-p},\dots,\rho_n\>$ of the target already considered above.

We are left to check that the face maps $d_i$ correspond to each other under the isomorphisms $\phi_p$.
Because the face maps are equivariant with respect to the
$\Aut_\M (S\hash D^{\hash n})$-action in both cases,
and the actions are transitive,
it is enough to check that the face maps agree for the simplices
$f_p$ and $\s_p=\phi_p(f_p)$. By definition, 
$$
d_if_p
=((S\hash D^{\hash n-p-1})\hash D,\,
\id_{S\hash D^{\hash n-p-1}}\hash b_{D^{\hash i},D}^{-1}\hash \id_{D^{\hash p-i}})
$$
while
$$
d_i\s_p=\<
\rho_{n-p},\dots,\widehat{\rho_{n-p+i}},\dots,\rho_n\>$$ is the simplex obtained by forgetting the $(i+1)$st arc.
In particular, we immediately have that $d_0(f_p)=f_{p-1}$ and $d_0(\s_p)=\s_{p-1}=\phi_{p-1}(f_{p-1})$ giving that the face maps agree in that case.

For the remaining face maps, note that $$\id_{S\hash D^{\hash n-p-1}}\hash b_{D^{\hash i},D}^{-1}\oplus \id_{D^{\hash p-i}}=T_{n-p+i-1}\circ \dots\circ T_{n-p}\colon S\hash D^{\hash n}\rar S\hash D^{\hash n}$$
as composition of Dehn twists $T_i$ of
Section~\ref{sec:braidedaction}.
We need to compute the image of $\rho_{n-p+1},\dots,\rho_n$ under this map. 
By Lemma~\ref{lem:Trho}, we have that for $1\le j\le i$, 
\begin{align*}
  T_{n-p+i-1}\circ \dots\circ T_{n-p}(\rho_{n-p+j})&=T_{n-p+i-1}\circ \dots\circ T_{n-p+j-1}(\rho_{n-p+j})\\
                                                   &=T_{n-p+i-1}\circ \dots\circ T_{n-p+j}(\rho_{n-p+j-1}) \\
                                                     &=\rho_{n-p+j-1}
\end{align*}
while for $i+1\le j\le p$, 
$$ T_{n-p+i-1}\circ \dots\circ T_{n-p}(\rho_{n-p+j})=\rho_{n-p+j}.$$
Hence $d_i(f_p)$ takes the arcs $\rho_{n-p+1},\dots,\rho_n$ to the
arcs  $$\rho_{n-p},\dots,\rho_{n-p+i-1},\rho_{n-p+i+1},\dots,\rho_n,$$
i.e.~precisely to the arcs of $d_i(\s_p)$. So we indeed have that
$\phi_{p-1}(d_i(f_p))=d_i(\phi_{p-1}(f_p))$, which finishes the
proof. 
      \end{proof}

\subsection{Coefficient systems}

Having identified the space of destabilizations with the semi-simplicial
set of disordered arcs in Proposition~\ref{prop:Wiso},
we can now input the connectivity computation of the disordered arc complex
of Section~\ref{sec:high-cnt}
into the general stability theorem of \cite{krannich19}.
To state the resulting stability theorem in full generality,
we need to introduce the notions of  (split) finite degree coefficient systems.
We follow \cite[Sec 4]{krannich19}, which generalizes  \cite[4.1-4]{RWW17} that unify the earlier definitions of  Dwyer for the
general linear groups \cite{Dwyer} and Ivanov for the mapping class groups
\cite{IvanovTwisted}.
(The papers \cite{krannich19,RWW17}  consider in addition abelian coefficient
systems,  but these are not relevant here, because the abelianization of the
mapping class group of surfaces of large enough genus is trivial by a theorem
of Mumford--Birman--Powell, see Lemma 1.1 in~\cite{harer83}.) 

\medskip

Fix a bidecorated surface $S=(S,m,\varphi)$,
and let $D$ be the bidecorated disk as above.
Definition~4.1 of
\cite{krannich19} becomes in our case: 
\begin{definition}
  A {\em coefficient system} for the groups
  $\Aut_\M(S\hash D^{\hash n})$ with respect to the stabilization by $D$ is
a collection of  $\Z[\Aut(S\hash D^{\hash n})]$-modules $M_n$ for
$n\ge 0$, together with  maps $s_n\colon M_n\to M_{n+1}$ that are
equivariant with respect to the stabilization map
$\Aut(S\hash D^{\hash n})\xrightarrow{\hash D} \Aut(S\hash D^{\hash n+1})$, 
satisfying the following condition:
\begin{equation}\label{coef-cond}
  T_{n+1}\in \Aut(S\hash D^{\hash n+2})\ \textrm{acts trivially on the
    image of } M_n \xrightarrow{s_{n+1}\circ s_n} M_{n+2}
\end{equation}
for $T_{n+1}$ the Dehn twist of Section~\ref{sec:braidedaction} with support the last
two disks in $S\hash D^{\hash n+2}$. 
\end{definition}
We will encode the data of  a coefficient system as a pair $(F,\s^F)$
with 
$$F\colon \M|_{S,D} \rar \operatorname{Mod}_\Z$$
a functor from the full subcategory of $\M$ on the objects $S\hash D^{\hash n}$
for $n\ge 0$ to abelian groups, where $M_n=F(S\hash D^{\hash
  n})$ with its $\Aut(S\hash D^{\hash n})$-action induced by $F$, and 
$$\s^F\colon F(-) \rar F(-\hash D)$$
is a natural transformation encoding the suspension maps $s_n$,
where we assume that $F(\id \hash T)$ acts trivially on the image of
$(\s^F)^2\colon F(-) \to F(-\hash D^{\hash 2})$ for $T$ the Dehn twist
supported on the added disks $D^{\hash 2}$.

\medskip

Given a coefficient system $F$, we define its {\em suspension} $\Sigma F\colon \M|_{S,D} \rar \operatorname{Mod}_\Z$ by $\Si F(-)=F(-\hash D)$ with
$$\s^{\Si F}\colon \Si F(-)=F(-\hash D) \xrightarrow{\s^F} F(-\hash D^{\hash 2}) \xrightarrow{\id\hash T} F(-\hash D^{\hash 2})=\Si F(-\hash D), $$
where one checks that the triviality condition \ref{coef-cond} is
satisfied with this choice of structure map $\s^{\Si F}$. (See
\cite[Def 4.4]{krannich19}.)

The structure map $\s^F$ induces a natural transformation $F\to \Si F$, called the {\em suspension map}. We define the {\em kernel} $\ker F$ and {\em cokernel} $\coker F$ to be the kernel and cokernel functors of that natural transformation. 
We call $F$ {\em split} if the suspension map is split injective in the category of coefficient systems. 

\begin{definition}\label{def:coef-system}\cite[Def 4.10]{RWW17} A coefficient system $F$ is
\begin{enumerate}
\item of {\em (split) degree $-1$ at $N$} if $F(S\hash(D^{\hash n}))=0$ for all $n\ge N$; 
\item of {\em degree $k\ge 0$ at $N$} if $\ker(F)$ has degree $-1$ at $N$ and $\coker(F)$ has degree $(k-1)$ at $(N-1)$; 
\item of {\em split degree $k\ge 0$ at $N$} if $F$ is split and $\coker(F)$ is of split degree $(k-1)$ at $(N-1)$.
  \end{enumerate}
\end{definition}
\begin{example}
  \leavevmode
  \begin{enumerate}\label{exmp:coef-systems}
  \item A coefficient system $F$ is of degree $0$ at $0$
    if and only if $\sigma^F$ is a natural isomorphism. This is in
    particular the case for constant coefficient systems. 
  \item The functor $F_k\colon\M\to  \operatorname{Mod}_\Z$ defined by
$$F_k(S)=H_1(S,\Z)^{\otimes k}$$
is a split coefficient system of degree $k$ at $0$. 
(This is essentially a result of Ivanov \cite[Sec 2.8]{IvanovTwisted},
who considers a version of the composite stabilization $\hash D^{\hash 2}$. 
See also \cite[Ex 4.3]{boldsen12} for the case $k=1$, and \cite[Lem 2.9]{Soulie} that proves this in a very general
set-up, though in the case of a braided groupoid acting over itself only.)
\item Given a $k$-connected space $X$,
  the coefficient system $F_n^k\colon\M\to\text{Mod}_\Z$ defined by
  $$
  F_n^k(S) = H_n(\operatorname{Map}(S/\partial S),X),
  $$
  which appears in the work of Cohen--Madsen~\cite{CohMad},
  is a coefficient system of degree $\floor{n/k}$ (see \cite[Ex 4.3]{boldsen12}).
  \end{enumerate}
\end{example}

\begin{remark}\label{rmk:different-category}
Although the above examples all makes sense in the different set-ups
considered in the literature, one should keep in mind that there are
variations in what precisely a finite degree coefficient system for
the mapping class groups of surfaces means in e.g.~the papers
\cite{IvanovTwisted,CohMad,boldsen12,RWW17} and \cite{krannich19}.
This is due to two facts: first, the definition of the coefficient
system depends on the category of surfaces considered and on the stabilization map(s) one works with,
and second, the triviality
condition \eqref{coef-cond} arising from Krannich's framework is
actually weaker  than the one used in earlier frameworks, see
e.g.~\cite[Rem 7.9]{krannich19}.

In addition, the paper 
\cite{GKRW19MCG} uses a homological
condition instead of a finite degree condition (see 5.5.1 in that
paper). The relationship between that condition and finite degree
conditions is discussed in \cite[Rem 19.11]{GKRW18cell}.
  \end{remark}

  \subsection{The stability theorem}

We are now ready to state our main theorem: 
   \begin{theorem}\label{thm:stab} Let $S=(S,m,\phi)$ be an object of
     $\M$ with $m$ odd, i.e.~such that $I_0,I_1$ are in the same boundary component. 
        Let $F\colon\M|_{S,D}\to\operatorname{Mod}_\Z$ be a coefficient system
        and write $F_n = F(S\hash D^{\hash n})$.
        The map
        $$
        H_i(\Aut_\M(S\hash D^{\hash n});F_n)
        \rar H_i(\Aut_\M(S\hash D^{\hash n+1});F_{n+1})
        $$ 
        is
        \begin{enumerate}
        \item an epimorphism
          for $i\leq\frac{n} 3$
          and an isomorphism for $i\leq\frac{n-3} 3$ if $F$ is constant.
        \item an epimorphism
          for $i\leq \frac{n-3k-2} 3$
          and an isomorphism for $i\leq\frac{n - 3k - 5} 3$
          if $F$ has degree $k$
          at $N\geq 0$ and $n>N$.
        \item an epimorphism
          for $i\leq\frac{n-k-2} 3 $
          and an isomorphism for $i\leq\frac{n-k-5} 3$
          if $F$ has split degree $k$
          at $N\geq 0$ and $n>N$.
        \end{enumerate}
      \end{theorem}

      \begin{remark}
We have stated the theorem in the case of an initial surface $S$ with
$I_0$ and $I_1$ in the same boundary component for simplicity. The
case of a surface $S'$ where the two intervals lie in different components is actually
also included in the statement, by writing $S'=S\hash D$ for $S$ of
the previous type, or considering $S'\hash D$ if $S'$ does not admit
such a decomposition. Indeed, as we have already seen in
Section~\ref{sec:category} (see Figure~\ref{fig:glueD}),
gluing in a disk exactly changes whether
$I_0$ and $I_1$ are in the same boundary or not. 
        \end{remark}

We will first show that the above results implies the two main
theorems stated in the introduction.

\begin{proof}[Proof of Theorems~\ref{thmintro:stab1}
  and~\ref{thmintro:stab2} from Theorem~\ref{thm:stab}]
Let $S_{0,r}^s$ be a surface of genus 0 with $r\ge 1$ boundary components and
$s$ punctures, and consider the associated object
$S=(S_{0,r}^s,1,\phi)$ of $\M$, with two marked intervals in the first
boundary component. 
Then $S\hash D^{\hash 2g}$ has the form  $(S_{g,r}^s,1+2g,\phi)$ while
$S\hash D^{\hash 2g+1}$ has the form
$(S_{g,r+1}^s,2+2g,\phi)$. Moreover, the maps
 $$S\hash D^{\hash 2g} \xrightarrow{\hash D} S \hash D^{\hash 2g+1}
 \xrightarrow{\hash D}S\hash D^{\hash 2g+2}$$
 precisely induce on automorphism groups in $\M$ the two maps appearing in Theorems~\ref{thmintro:stab1}
 and~\ref{thmintro:stab2}.
 
The fact that the first map is always injective in homology follows from the fact
that postcomposing the map  $S_{g,r}^s\to S_{g,r+1}^s$,
defined by the sum $\hash D$, with the map $S_{g,r+1}^s \to
S_{g,r+1}^s\cup_{S^1}D^2\simeq S_{g,r}^s$  filling in one of the newly created
boundary component, is homotopic to the identity. 
Now Theorem~\ref{thm:stab}(a) gives that the  map
$$H_i(\Aut_{\M}(S\hash D^{\hash 2g})) \xrightarrow{\hash D} H_i(\Aut_{\M}(S \hash D^{\hash 2g+1}))$$
is surjective  for
$i\leq\frac{2g} 3$ in homology with constant coefficients.  Given that the map is always injective, we get an isomorphism 
in that same range, proving the first part of
Theorem~\ref{thmintro:stab1}. Applying (b) and (c) instead gives
Theorem~\ref{thmintro:stab2} for the first map.

For the second map, we now apply Theorem~\ref{thm:stab} in the case $n=2g+1$, but
in that case, there is no additional argument for injectivity, so the
bounds translate directly to surjectivity and isomorphism bounds. 
  \end{proof}

      \begin{proof}[Proof of Theorem~\ref{thm:stab}]
        Proposition~\ref{prop:Wiso} together with Theorem~\ref{thm:disord-cnt}
        give that $W_n(S,D)_\bullet$ is $\left(\frac{2g+\nu-5} 3\right)$-connected, for $g$ the genus of $S\hash D^{\hash n}$ and $\nu=1$ if $I_0$ and $I_1$ are in the same boundary component of  $S\hash D^{\hash n}$, which is the case precisely when $n$ is even, and $\nu=2$ otherwise.  
    The surface  $S\hash D^{\hash n}$  has genus greater than or equal to the genus of  $D\hash D^{\hash n}$, that is $\frac{n}{2} $ if $n$ is even and $\frac{n-1}{2}$ if $n$ is odd (see Lemma~\ref{lem:surface-type}).
Hence $2g+\nu\ge n+1$ is both cases, and $W_n(S,D)_\bullet$ 
is at least $\left(\frac{n-4} 3\right)$-connected.

Now $W_n(S,D)_\bullet$  is the semi-simplicial set denoted $W^{\text{RW}}(S\hash D^{\hash n})_\bullet$ in \cite{krannich19} (see Definition~7.5 in that paper). By Lemma 7.6 in the same paper, using Proposition~\ref{prop:Mmonoid}, this semi-simplicial set has the same connectivity as the semi-simplicial space  $W(S\hash D^{\hash n})_\bullet$ of \cite{krannich19}, which by Remark~2.7 of that paper determines the connectivity assumption of Theorem A in that paper: the canonical resolution of the assumption of the theorem is $m$-connected, if and only if the space $W(S\hash D^{\hash n})_\bullet$ is $(m-1)$-connected.
Given that  $W(S\hash D^{\hash n})_\bullet$ is $\left(\frac{n-4}
  3\right)$--connected, we have that the canonical resolution of
 is $\left(\frac{n-4+3} 3\right)$--connected. Hence we can apply
 \cite[Thm A]{krannich19} with $k=3$ and grading $g_{\M}:\M|_{S,D}\to
 \N$ given by $g_{\M}(S\hash D^{\hash n})=n-2$; see also \cite[Rem
        2.24]{krannich19}, where we can take $m=4$.
The theorem, with the improvement given by (i) in the remark,  then gives that
        $$
        H_i(\text{Aut}_{\M}(S\hash D^{\hash n});\Z )\rar
        H_i(\text{Aut}_{\M}(S\hash D^{\hash n+1});\Z )
        $$
        is an isomorphism for $i\leq\frac{n-3} 3$ and
        an epimorphism for $i\leq \frac{n}3$, giving the stated result
        in the case of constant coefficients. 
        For a coefficient system $F$ of degree $k$ at $N$,
        \cite[Thm C]{krannich19} gives that
          $$
        H_i(\text{Aut}_{\M}(S\hash D^{\hash n});F_n )\rar
        H_i(\text{Aut}_{\M}(S\hash D^{\hash n+1});F_{n+1} )
        $$
        is an isomorphism for $i\leq\frac{n-3k-5} 3$ and
        an epimorphism for $i\leq \frac{n-3k -2} 3 $
        for $n>N$, improved to an isomorphism for $i\leq\frac{n-k-5} 3$ and
        an epimorphism for $i\leq \frac{n-k -2} 3 $ if $F$ is split. 
      \end{proof}

      \begin{remark}[Optimality of the stability bounds]\label{rem:optimal2}
Combining the two maps in Theorem~\ref{thmintro:stab1}, 
we obtain that the genus stabilization 
\begin{equation*}
  H_i(\Gamma(S_{g,r}^s);\Z )\to H_i(\Gamma(S^s_{g+1,r}); \Z) 
\end{equation*}
is an  epimorphism when $i\leq\frac{2g} 3$
and an  isomorphism when $i\leq\frac{2g-2} 3$.
The slope $\frac{2}{3}$ is known to be optimal by a computation of
Morita \cite{Morita}, with optimal isomorphism range
since for instance $H_1(\Gamma (S_{2,r});\Z )\to H_1(\Gamma (S_{3,r});\Z )$
is not injective as the source is isomorphic to $\Z /12$ and
the target is trivial, see e.g.~\cite[Theorem 5.1]{korkmaz}.
Our combined genus epimorphism range, on the other hand, falls short
of the range  $i\leq\frac{2g+1} 3$, as given in
\cite{GKRW19MCG}, a range that is optimal by Morita's computation
(see Theorem B (i) of \cite{GKRW19MCG}).

        Our results for twisted coefficients are most easily compared with those
        of Boldsen  \cite[Thm 3]{boldsen12}, whose coefficient systems are 
        coefficient systems of finite split degree in our sense, 
        though with a stricter triviality condition upon double stabilization.
        For these coefficient systems, he obtains slightly better
        ranges, with improvement $+\rfrac{2}{3}$ for the first
        map and $+\rfrac{5}{3}$ for the second.
The papers \cite{RWW17,GKRW19MCG} only consider genus stability. 
        In \cite{RWW17}, the stability slope obtained is only
        $\frac{1}{2}$, while in \cite[Sec 5.5.1]{GKRW19MCG}, the
        finite degree condition is replaced by a more general homological
        condition that applies to some finite coefficient systems
        \cite[Sec 19.2]{GKRW18cell}. In the particular
        case of the $k$th tensor
        power of the first homology of the surface, they do however
        only get the epimorphism range $i\leq\frac{ 2g-2k+1} 3$
        and  isomorphism range $i\leq\frac{ 2g-2k-2} 3$, see Example 5.22
        in that paper. 
      \end{remark}

\section{Braiding and homological stability for groups}
\label{sec:braiding}
In order to use the framework of Krannich~\cite{krannich19}
to prove homological stability for a sequence of groups,
one needs the structure of an ``$E_1$-module over an $E_2$-algebra''.
We give in Proposition~\ref{prop:YB} below a simple way to construct such a module structure,
in terms of Yang--Baxter operators. 
Compared to earlier approaches to homological stability such as~\cite{RWW17},
which Krannich's work generalizes,
this has the advantage of being very lightweight.
Instead of having to provide the structure of a braiding on the monoidal category whose automorphism
groups one is interested in,
it suffices to provide a single morphism satisfying a simple
equation.

Our main example of a Yang--Baxter operator is the
inverse Dehn twist $T_1^{-1}\in \Aut_{\M}(D\hash D)$, defined in
Section~\ref{sec:braidedaction} and used to prove our main result.
In Section~\ref{sec:notbraided}, we show that this Yang--Baxter operator is not part
of a braided monoidal structure on the category $\M$, but gives
instead a twisted version of such a structure.

\subsection{Yang--Baxter operators and braid groupoid actions}\label{sec:YB}
Let $\mathcal X = (\mathcal X,\oplus ,\mathbbm 1)$ be a monoidal category.
A {\em Yang--Baxter operator}  in $\mathcal X$
is a pair $(X,\tau )$ consisting of an object $X\in\mathcal X$
and a morphism $\tau\in \Aut_{\mathcal C}(X\oplus X)$,
satisfying the Yang--Baxter equation
$$
(\tau\oplus 1)(1\oplus\tau)(\tau\oplus 1)= (1\oplus\tau)(\tau\oplus 1)(1\oplus\tau ) \in \Aut_{\mathcal C}(X\oplus X\oplus X),
$$
where we suppress associators from the notation.

Yang--Baxter operators are closely related to the braid groupoid: 
Recall from Section~\ref{sec:braidedaction} the braid groupoid $\BB$, with objects
the natural numbers and only non-trivial morphisms
$\Aut_\BB(n)=B_n$.  
A variant of the coherence theorem for braided
monoidal categories says that the category of strong monoidal functors
from the braid groupoid into $\mathcal X$ is equivalent to a naturally
defined category of Yang--Baxter operators in $\mathcal X$
\cite[Prop 2.2]{joyal-street}.\footnote{
In other words, the pair consisting of the braid groupoid $\BB$
and the Yang--Baxter operator $\sigma_1\in\text{Aut}_\BB (2)$,
is the initial monoidal category
with a distinguished Yang--Baxter element.}
To a Yang--Baxter operator $(X,\tau )$ in $\mathcal X$,
this equivalence associates
the strong monoidal functor $\Phi_{X,\tau}\colon\BB\to\mathcal X$ given by
$\Phi_{X,\tau} (n) = X^{\oplus n}$ on objects,
and on morphisms by letting
$$
\Phi_{X,\tau}\colon B_n\to\Aut_{\mathcal X}(X^{\oplus n})
$$
send the $i$th standard generator $\sigma_i$ to
$\text{id}_{X^{\oplus i-1}}\oplus\tau\oplus\text{id}_{X^{\oplus n-i-1}}$,
where the required maps
$\Phi_{X,\tau} (m)\oplus \Phi_{X,\tau} (n)\to \Phi_{X,\tau} (m+n)$
are given by the monoidal structure of  $\mathcal X$.

\smallskip

Suppose now that the monoidal category
$\mathcal X$ acts on a category $\mathcal M$
via a functor $\mathcal M\times\mathcal X\to\mathcal M$,
which we also denote by $\oplus$, compatible with the
monoidal sum in $\mathcal X$.
The following result shows that the choice of a Yang--Baxter operator
defines an action of the braid groupoid $\BB$ on $\mathcal M$, and hence is
appropriate data to apply the stability framework of \cite{krannich19}:

\begin{proposition}\label{prop:YB}
  Let $(\mathcal X,\oplus ,\mathbbm 1)$ be a monoidal category
  with $\tau\in \Aut_{\mathcal X}(X\oplus X)$
  a Yang--Baxter operator in $\mathcal X$.
  Suppose $\mathcal X$ acts on a category $\mathcal M$.
  Then there is an action of the braid groupoid
  $$
  \alpha_\tau\colon \mathcal M\times\BB\to\mathcal M
  $$
  given on objects by $\alpha_\tau(A,n)= A\oplus X^{\oplus n}$
  and determined on morphisms by $$\alpha_\tau(f,\sigma_i )=f\oplus\text{id}_{X^{\oplus i-1}}\oplus\tau\oplus\text{id}_{X^{\oplus n - i -1}},$$
  for $\sigma_i$ the $i$th elementary braid in $B_n$. 
  Furthermore, taking classifying spaces this endows $B\mathcal M$
  with the structure of an $E_1$-module over the $E_2$-algebra $B\BB$.
\end{proposition}

Note that if we are interested in homological stability for
stabilization by $X$ for the automorphism groups
$G_n:=\Aut_{\mathcal M}(A\oplus X^{\oplus n})$
for some object $A$ of $\mathcal M$,
only  the full subcategory $\mathcal M_{A,X}\subseteq\mathcal M$
  spanned by objects of the form $A\oplus X^{\oplus n}$, is relevant. 
  So for stability purposes, it is enough to consider the
  subfunctor
  $$
  \alpha_\tau\colon \mathcal M_{A,X}\times\BB\to\mathcal M_{A,X}. 
  $$
In fact, to make sure that the structure of $E_1$--module over the
$E_2$--algebra $B\BB$ is graded, one can even replace the category
$\mathcal M_{A,X}$ by a category with objects the natural
    numbers and setting $\text{Aut}(n) = \text{Aut}_{\mathcal M}(A\oplus
    X^{\oplus n})$, avoiding any potential issue  coming  from unwanted
 equalities $A\oplus X^{\oplus n} = A\oplus X^{\oplus m}$ for $m\neq
 n$.
 
\begin{proof}
The functor
$
\alpha_\tau\colon \mathcal M\times\BB\to\mathcal M
$
is defined as the composite functor  $$\alpha(-,-)= (-)\oplus \Phi_{X,\tau} (-),$$ for
$\Phi_{X,\tau}\colon \BB\to \mathcal X$ as above.  
The result follows from \cite[Lem 7.2]{krannich19} because $\alpha$
makes $\mathcal M$ into a module over $\BB$  and $\BB$ is braided monoidal.
\end{proof}

\begin{example}\label{ex:braided}
If $\mathcal{X}=(\mathcal X,\oplus ,\mathbbm 1)$ admits a braiding
$b$,  then $\tau=b_{X,X}\in \Aut_{\mathcal X}(X\oplus X)$
is a Yang--Baxter operator for any object $X$. For
$\mathcal X$ a groupoid acting on itself
or $\mathcal X$ acting on a category $\mathcal M$,
this  recovers the basic set-up for homological stability
of the paper \cite{RWW17}, or Section 7 of \cite{krannich19}.
  \end{example}

  \begin{example}[Mapping class groups of surfaces]\label{exmp:mcg-YB}
    As explained above, a Yang--Baxter operator $\tau\in
    \Aut_{\mathcal X}(X\oplus X)$
    gives in particular a collection of homomorphisms
    $\Phi_{X,\tau}\colon B_n\to\Aut_{\mathcal X}(X^{\oplus n})$
    from the braid groups to the automorphism group of $n$ copies of $X$.
    There are two standard ways to embed braid groups in mapping class
    groups of surfaces, and we explain here how they both come from
    Yang--Baxter elements in appropriate categories of surfaces. 
 \begin{enumerate}
 \item  Let $\M$ be the category of bidecorated surfaces of
   Section~\ref{sec:category}.
   As explained in Section~\ref{sec:braidedaction}, 
   the  Dehn twist
   $T\in \Aut_{\M}(D\hash
   D)\cong\pi_0\operatorname{Homeo}_\partial(S^1\x I)\cong \Z$, or its
   inverse $T^{-1}$, 
   is a Yang--Baxter operator.
   The associated map $\Phi_{D,T}\colon B_n\to \Aut_{\M}(D^{\hash n})$
   is the embedding of braid group in the mapping class
   groups of $S_{g,1}$ (when $n=2g+1$) and of $S_{g,2}$ (when
   $n=2g+2$) associated to Dehn twists along the chain of embedded
   curves in the surfaces described in Lemma~\ref{lem:chain}. This
   embedding goes back at least to the work of Birman
   and Hilden \cite{BH1,BH2}.
 \item
   Let ${\mathbf M}_1$ denote instead the category of surfaces decorated
   by a single interval, with monoidal structure $\oplus$ defined just as in
   the case of $\mathbf M_1$ but gluing only along one interval. Then $\mathbf M_1$ in
   braided monoidal, see \cite[Sec 5.6.1]{RWW17}. Hence by
   Example~\ref{ex:braided}, for any
   object $X$ of $\mathbf M_1$, we have a Yang--Baxter element $\tau_X\in
   \Aut_{\mathbf M_1}(X\oplus X)$. For $X = S_{1,2}$,
   this can be used to prove genus stabilization (albeit with the suboptimal slope $\rfrac 1 2$),
   and in the case $X=S^1\times I$ marked
   by an interval in one of its boundary components, we have that
   $X^{\oplus n}$ has underlying surface an $n$-legged pair of
   pants $D^2\backslash (\sqcup_n \mathring{D}^2)$ and the associated morphism 
   $$
   \Phi_{X,\tau_X}\colon
   B_n\to   \Aut_{\mathbf M_1}(X^{\oplus n})
   = \pi_0\operatorname{Homeo}_\partial(D^2\backslash (\sqcup_n\mathring{D}^2))
   $$ 
is the standard embedding of the braid group as the subgroup of the
mapping class group of the multi-legged pants that does not twist the
legs, see e.g.~\cite[Sec 5.6.1]{RWW17}.
 \end{enumerate}
 We will show in Proposition~\ref{prop:nobraiding} below that the
 Yang--Baxter operator $T$ of the first example, in
 the category $\M$, does not come from a braiding in $\M$. 
 \end{example}

\subsection{Homological stability from Yang--Baxter elements}

Suppose we are given the data of a monoidal category
$(\mathcal X,\oplus ,\mathbbm 1)$ acting on a category $\mathcal M$,
along with a choice of stabilizing object $X\in\mathcal X$
and Yang--Baxter operator $\tau\in\text{Aut}_{\mathcal X}(X\oplus X)$.
Proposition~\ref{prop:YB} above allows to apply \cite[Thm
A]{krannich19}, which in this case says that for any
$A\in\mathcal M$, there  is a sequence of simplicial spaces $W_n(X,A)_\bullet$, for $n\ge
0$, so that if $W_n(X,A)$ is highly-connected
for large $n$, then the sequence
$$
\text{Aut}_{\mathcal M}(A)\xrightarrow{{-}\oplus X}
\text{Aut}_{\mathcal M}(A\oplus X)\xrightarrow{{-}\oplus X}
\text{Aut}_{\mathcal M}(A\oplus X\oplus X)\xrightarrow{{-}\oplus X}\cdots
$$
satisfies homological stability. Theorem B of the same paper gives in
addition a stability statement with twisted coefficients.
Under an injectivity assumption of the form of
Proposition~\ref{prop:Mmonoid}, this simplicial space is homotopy
discrete, and modeled by the space of destabilizations
as described in Definition~\ref{def:WS}.

\begin{remark}
  The fact that $(X,\tau )$ is a Yang--Baxter operator
  is precisely what is needed
  for the collection of sets $W_n(A,X)_p$ and maps
  $d_i\colon W_n(A,X)_p\to W_n(A,X)_{p-1}$,
  defined as in Definition \ref{def:WS},
  to assemble into a semi-simplicial set;
  indeed, the Yang--Baxter equation implies the necessary
  simplicial identities.
\end{remark}

\smallskip

For a fixed monoidal category
$\mathcal X = (\mathcal X,\oplus ,\mathbbm 1)$
acting on a category $\mathcal M$,
and a stabilizing object $X\in\mathcal X$,
the choice of Yang--Baxter
element will not affect the stabilizing map, but it will affect the
spaces $W_n(X,A)_\bullet$. 
The identity map $1\in\text{Aut}_{\mathcal X}(X\oplus X)$ is a trivial
choice of Yang--Baxter operator.
But, as is to be expected, this trivial twist is not useful for proving
stability:

\begin{proposition}
 Let $\mathcal X,\mathcal M$,  $A$ and $X$ be as above. If we choose
 the Yang--Baxter operator $\tau\in \Aut_{\mathcal X}(X\oplus X)$
 to be the identity element,
 then the semi-simplicial set $W_n(A,X)_\bullet$ is connected if and only if
  the map
  $$
  G_{n-1}=\text{Aut}_{\mathcal M}(A\oplus X^{\oplus n-1})\xrightarrow{{-}\oplus X}
  \text{Aut}_{\mathcal M}(A\oplus X^{\oplus n})=G_n
  $$
 is an isomorphism.
\end{proposition}

\begin{proof}
  If $\tau$ is the identity element,
  all face maps $d_i$ are equal to the canonical map $G_n/G_{n-p-1}\to G_n/G_{n-p}$.
  In particular, the vertices of any $p$-simplex are all equal,
  so the semi-simplicial set $W_n(A,X)_\bullet$ is isomorphic to a disjoint union
  of semi-simplicial sets, one for each $0$-simplex. The result
  follows from the fact that the set of $0$-simplices is precisely
the quotient  $G_n/G_{n-1}$. 
\end{proof}
In fact,
Barucco proved in his master thesis a result that translates to the following
stronger statement (stated in the thesis in the context of a groupoid acting on itself,
i.e.~$\mathcal M=\mathcal X$):

\begin{lemma}\cite[Lem 3.1]{Barucco}
The space $W_n(A,X)$ is connected if and only if $1^{\oplus n-2}\oplus
\tau$ and $G_{n-1}\oplus 1$ together generate $G_n=\Aut(A\oplus
X^{\oplus n})$.  
\end{lemma}

The connectivity of the semi-simplicial set $W_n(A,X)$ (or of the associated simplical complex defined in \cite[Def 2.8]{RWW17}) can be thought of as a measure a form of {\em higher generation} of the group $G_n$ by the cosets of the subgroups $G_{n-p}$ for $p\ge 1$ and braid subgroups generated by the chosen Yang--Baxter element $t$,
in a way similar to the notion of higher generation for a family of subgroups of a group defined in \cite[2.1]{AbeHol}.

\subsection{Braidings and bidecorated surfaces}\label{sec:notbraided}
We show in this section that the Yang--Baxter operator $T$ on the bidecorated
disk $D$ in the groupoid $\M$ does not come from a braiding on the
subcategory of $\M$ generated by the disk. In fact, we will show that
this subcategory does not admit a braiding.

\medskip

Let $D=(D^2,1,\id)$ be the standard bidecorated disk of Section~\ref{sec:category},
where we recall that $X_1=D^2$.
We define a ``rotated'' bidecorated disk $\overline D=(D^2,1,r_\pi)$,
where $r_\pi$ is the rotation of $\del X_1 = \del D^2$ by $\pi$ radians,
which has the effect of interchanging the intervals $I_0$ and $I_1$.
Rotating all of $D^2$
by $\pi$ then induces a morphism $\iota\colon D\to \overline D$ in $\M$,
and likewise morphisms
$$
\iota^{\hash m}\colon D^{\hash m}\rar
\overline D^{\hash m}
$$
for every $m\ge 1$, each which we will by abuse of notation also denote by $\iota$.
The morphism $\iota$ can be identified with the hyperelliptic involution
of the underlying surface depicted in Figure~\ref{fig:hyperelliptic-involution}
for the two cases $m = 2g$ and $m = 2g+1$,
where in the latter case the boundary components
are exchanged by $\iota$.
\begin{figure}
  \def\svgwidth{0.8\textwidth}
\begingroup%
  \makeatletter%
  \providecommand\color[2][]{%
    \errmessage{(Inkscape) Color is used for the text in Inkscape, but the package 'color.sty' is not loaded}%
    \renewcommand\color[2][]{}%
  }%
  \providecommand\transparent[1]{%
    \errmessage{(Inkscape) Transparency is used (non-zero) for the text in Inkscape, but the package 'transparent.sty' is not loaded}%
    \renewcommand\transparent[1]{}%
  }%
  \providecommand\rotatebox[2]{#2}%
  \newcommand*\fsize{\dimexpr\f@size pt\relax}%
  \newcommand*\lineheight[1]{\fontsize{\fsize}{#1\fsize}\selectfont}%
  \ifx\svgwidth\undefined%
    \setlength{\unitlength}{85.03937008bp}%
    \ifx\svgscale\undefined%
      \relax%
    \else%
      \setlength{\unitlength}{\unitlength * \real{\svgscale}}%
    \fi%
  \else%
    \setlength{\unitlength}{\svgwidth}%
  \fi%
  \global\let\svgwidth\undefined%
  \global\let\svgscale\undefined%
  \makeatother%
  \begin{picture}(1,0.45786667)%
    \lineheight{1}%
    \setlength\tabcolsep{0pt}%
    \put(0,0){\includegraphics[width=\unitlength,page=1]{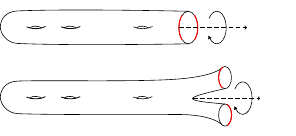}}%
    \put(0.34887332,0.35882954){\makebox(0,0)[lt]{\lineheight{1.25}\smash{\begin{tabular}[t]{l}$\cdots$\end{tabular}}}}%
    \put(0.34887332,0.12110242){\makebox(0,0)[lt]{\lineheight{1.25}\smash{\begin{tabular}[t]{l}$\cdots$\end{tabular}}}}%
  \end{picture}%
\endgroup%

  \caption{The hyperelliptic involutions $\iota$ of $S_{g,1}$ and $S_{g,2}$}\label{fig:hyperelliptic-involution}
\end{figure}
The morphism $\iota$ induces an identification
\begin{align*}
  \text{Aut}_{\M}(D^{\hash m})
  &\tox{\cong} \text{Aut}_{\M}(\overline D^{\hash m})
  \\
  f &\mapstox{\phantom{\sim}} \iota \circ f \circ \iota^{-1}
\end{align*}
In order to precisely state the failure of $T$ to extend to a braiding,
we will also need the identification
\begin{align*}
 I\colon \text{Aut}_{\M}(D^{\hash m})
  &\tox{\cong} \text{Aut}_{\M}(\overline D^{\hash m})\\
   f &\mapstox{\phantom{\sim}}  f
\end{align*}
that comes from the fact that an element $f\in\text{Aut}_{\M}(D^{\hash m})$
is just a mapping class for the underlying surface of $D^{\hash m}$,
which is the same as the underlying surface of $\overline D^{\hash m}$,
so $f$ can just as well be viewed as an element of
$\text{Aut}_{\M}(\bar D^{\hash m})$.
In contrast with the identification induced by $\iota$,
the second identification is  ``external'',
in the sense that it does not come from a morphism in $\M$.

Viewing $\iota$ as a diffeomorphism of the underlying surface $X_m$ of
$D^{\hash m}$ that does not fix the boundary,
and specifically exchanges the marked points $b_0 = I_0(\rfrac 1 2)$
and $b_1 = I_1(\rfrac 1 2)$, we see that it takes the isotopy class of
arc $\rho_i$ of Section~\ref{sec:disk}
to the reversed arc $\overline{\rho_i}$. 
We will use in the proof of the following result that the  homotopy classes $\rho_i$
generate the fundamental groupoid of $X_m$ based at the points $b_0,b_1$.\footnote{
  As a full subgroupoid of the ordinary fundamental groupoid of $X_m$,
  this groupoid is the one spanned by the objects corresponding to the
  points $b_0,b_1\in X_m$.
  } The mapping class 
$\iota$  is in fact
completely determined by the fact that $\iota(\rho_i) = \overline{\rho_i}$.

\newpage

\begin{proposition}\label{prop:nobraiding}
  Let $\mathbf D\subset\M$ denote the full monoidal subcategory
  generated by $D$.
  \begin{enumerate}[label=(\roman*)]
  \item The monoidal category $\mathbf D$ does not admit a braiding.
    In particular, the monoidal functor $$\Phi\colon (\BB,\oplus)\to
    (\mathbf D,\hash)\subset (\M,\hash)$$ does not come from a braiding on
    $\mathbf D$.
  \item Let $f\in \Aut_\M(D^{\hash m})$ and $g\in \Aut_\M(D^{\hash n})$, and
    put $\beta_{m,n}=\Phi (b_{m,n})$, where the block braid $b_{m,n}$
    is the braid which passes the last $n$ strands over the first $m$ strands.
    Then 
    $$
    \beta_{m,n} \circ (f\hash g)  \circ \beta_{n,m}^{-1}
    =
    \begin{cases}
      g \hash (\iota^{-1}\circ  f \circ \iota)
      &\text{if } n \text{ is odd},\\
      g\hash f&\text{else,}
    \end{cases}
    $$
    for $\iota\colon D^{\hash m}\to \overline D^{\hash m}$ the involution defined
    above, and where $f$ in the rightmost expression is the map $f$ considered as
    an element of $\Aut_\M(\overline D^{\hash m})$ via the isomorphism
    $I$ defined above.
  \end{enumerate}
\end{proposition}

\begin{proof}
We start by proving (ii). It is enough to check the statement when $f$
and $g$ are Dehn twists, as those 
generate the mapping class groups. 
Note that if $c$ is a curve in the underlying surface $X_{m+n}$
of $D^{\hash m + n}$, and $T_c$ denotes the Dehn twist along $c$, then conjugating $T_c$  by a diffeomorphism $\phi$ of the surface gives 
$$\phi\circ T_c\circ \phi^{-1}=T_{\phi(c)}.$$
Recall further that the isotopy class of a Dehn twist $T_c$ depends only on the free homotopy class of the curve $c$.
We are therefore to compute the images of curves in $D^{\hash m}$ and $D^{\hash n}$ under the map $\beta_{m,n}$, as free homotopy classes. A curve $c$ can be written, up to free homotopy, as a concatenation of the arcs $\rho_i$ and their inverses $\overline\rho_i$,
as the homotopy classes of these arcs
generate the fundamental groupoid of the surface $X_{m+n}$ based at $b_0,b_1$.
In particular, write
  \begin{equation}\label{equ:c}
    c \simeq \rho_{i_1}*\overline\rho_{i_2}*\rho_{i_3}\dots *\overline\rho_{i_k}.
  \end{equation}
  The mapping class $\beta_{m,n}$ can be written as the composition
  $$
  \beta_{m,n}=  (T_n\circ \dots \circ T_{m+n-1})\circ  \dots\circ (T_2\circ \dots \circ T_{m+1}) \circ (T_1\circ \dots \circ T_m)
  $$
  and hence we can compute the image of each $\rho_i$ using Lemma~\ref{lem:Trho}.
  For $r>0$, denote by $T_{i, i+r}$ the composition of Dehn twists $T_i\circ T_{i+1}\circ \dots\circ T_{i+r}$. 
  Note first that
  $$T_{i, j}(\rho_{j+1}) \simeq T_{i, j-1}(\rho_{j}) \simeq \dots \simeq \rho_i.$$
  From this, it follows that for $i\ge 1$, 
  \begin{align*}
    \beta_{m,n}(\rho_{m+i})&\simeq (T_{n, m+n-1})\circ \dots\circ  (T_{1, m})(\rho_{m+i}) \\
                           &\simeq (T_{n, m+n-1})\circ \dots\circ  (T_{i, m+i-1})(\rho_{m+i}) \\
                           &\simeq (T_{n, m+n-1})\circ \dots\circ  (T_{i+1, m+i})(\rho_{i}) \\
    &\simeq\rho_i. 
  \end{align*}
  On the other hand, for $i\le k\le j$, we have
  $$T_{i, j}(\rho_{k})\simeq T_{i, k}(\rho_{k})\simeq T_{i, k-1}(\rho_k*\overline\rho_{k+1}*\rho_k)\simeq \rho_i*\overline\rho_{k+1}*\rho_i,$$
  from which we can deduce that for $i\le m$, 
  \begin{align*}
    \beta_{m,n}(\rho_i)&\simeq  (T_{n, m+n-1})\circ \dots\circ  (T_{1, m})(\rho_{i}) \\
                       &\simeq  (T_{n, m+n-1})\circ \dots\circ  (T_{2, m+1})(\rho_1*\overline\rho_{i+1}*\rho_1) \\
                       &\simeq  (T_{n, m+n-1})\circ \dots\circ  (T_{3, m+2})(\rho_1*\overline\rho_2*\rho_{i+2}*\overline\rho_{2}*\rho_1)\\
                       & \simeq  \cdots \\
                       &\simeq \rho_1*\iota(\rho_2)*\dots* \iota^{n-1}(\rho_{n})*\iota^n(\rho_{i+n})*\iota^{n-1}(\rho_{n})*\dots*\iota(\rho_{2})*\rho_1
  \end{align*}
  since $\iota^j(\rho_i)$ is $\rho_i$ when $j$ is even and $\overline\rho_i$ when $j$ is odd.

  If the curve $c$ lies in the last $n$ disks  $D^{\hash n}$ inside  $D^{\hash m+n}$, it can be written as a product \eqref{equ:c} with each $i_j>m$. Then the above computation gives that
  $$
  \beta_{m,n}(c)\simeq \rho_{i_1-m}*\overline\rho_{i_2-m}*\rho_{i_3-m}*\dots *\overline \rho_{i_k-m},
  $$
  that is, $c$ is mapped to the corresponding curve in the {\em first} $n$ disks  $D^{\hash n}$ inside  $D^{\hash n+m}=D^{\hash m+n}$.

  If the curve $c$ instead lies in the first $m$ disks  $D^{\hash m}$ inside  $D^{\hash m+n}$, it can be written as a product \eqref{equ:c} with each $i_j\le m$. Then the above computation gives that
  \begin{align*}
    \beta_{m,n}(c)&\simeq \rho_1*\iota(\rho_2)*\dots* \iota^{n-1}(\rho_{n})*\iota^n(\rho_{i_1+n})*\iota^{n+1}(\rho_{i_2+n})*\dots*\iota^{n+1}(\rho_{i_k+n})\\
                         & \hspace{6.66cm} *\iota^{n}(\rho_{n})*\dots*\iota^2(\rho_{2})*\iota(\rho_1)\\
                  &\simeq  \iota^n(\rho_{i_1+n})*\iota^{n+1}(\rho_{i_2+n})*\iota^{n}(\rho_{i_3+n})\dots*\iota^{n+1}(\rho_{i_k+n}) \\
                  &\simeq \iota^n ( \rho_{i_1+n}*\overline\rho_{i_2+n}*\rho_{i_3+n}\dots *\overline\rho_{i_k+n})
  \end{align*}
  Hence $c$ is mapped to the curve $\iota^n(c)$ in the last $m$ disks $D^{\hash m}$ inside  $D^{\hash n+m}=D^{\hash m+n}$, from which the statement follows.

  \smallskip

  We are left to prove (i). To see that the images $\beta_{m,n}$ of
  block braids under $\beta$  do not define a braiding in $\mathbf D$,
 using (ii) it is enough to find a curve $c$ in $D^{\hash m}$ for some $m$
 so that $\iota (c)\not\simeq c$, and such curves are plentiful. 
  The same argument shows that the inverses $\beta_{m,n}^{-1}$
  likewise do  not define a braiding.

  Now suppose that $\tilde \beta$ is a braiding on $\mathbf D$. The braiding is determined by
  $\tilde\beta_{1,1}\in\Aut_\M (D^{\hash 2})\cong \mathbb Z$,  a group
  generated by the Dehn twist $T_1$.
  We have excluded the possibilities $\tilde\beta_{1,1} = T_1^{\pm
    1}$, and $\tilde\beta_{1,1} = \id$ is similarly ruled out using
  now the fact that curves are not moved at all by the identity.
 So  assume that  $\tilde\beta_{1,1} = T_1^k$, with $|k|> 1$.
  Then $\tilde\beta_{2,1} = T_1^kT_2^k$ would have to satisfy
  $T_1^kT_2^k(a_1) = a_2$ in order for naturality to hold, where $T_i$
  is the Dehn twist along the curve $a_i$ as in Section~\ref{sec:braidedaction}.
  Applying Proposition 3.2 in~\cite{MCG-primer} twice, we get 
 that the intersection number $i(a_2,T_2^k(a_1)) =
  i(T_1^k(T_2^k(a_1)),T_2^k(a_1)) = |k|i(a_1,T_2^k(a_1)^2 =
  |k|^2i(a_1,a_2)^4 = |k|^2$. On the other hand, using Proposition 3.4 in~\cite{MCG-primer}
  we obtain 
  $$
  |k|^2 = i(a_2,T_2^k(a_1)) = |i(T_2^k(a_1),a_2)  - |k|i(a_1,a_1)i(a_1,a_2)|\leq i(a_1,a_2) = 1,
  $$
  where we have also used that $i(a_1,a_1) = 0$.  This contradicts our
  assumption of $\tilde \beta_{1,1}$. 
\end{proof}

  \bibliographystyle{amsalpha}
      \bibliography{sources}
\end{document}